\documentclass{amsart}
\usepackage{amssymb, amsmath, amscd, latexsym, mathrsfs}

\usepackage{tikz} 
\usetikzlibrary{calc,matrix,arrows,trees,positioning} 
\usepackage[all]{xy}

\usepackage{xr-hyper}

\usepackage[pagebackref=true,pdfnewwindow=true,pdftex]{hyperref}
\hypersetup{colorlinks, linkcolor=blue, citecolor=red}

\newtheorem{thm}{Theorem}[section]

\newtheorem{cor}[thm]{Corollary}
\newtheorem{lem}[thm]{Lemma}

\newtheorem{prop}[thm]{Proposition}

\theoremstyle{definition}
\newtheorem{defn}[thm]{Definition}

\theoremstyle{remark}

\newtheorem{rem}[thm]{Remark}

\numberwithin{equation}{section}

\newcommand{\Cat}{\mathsf{Cat}}

\newcommand{\FF}{\mathbb{F}} 							% finite sets
\newcommand{\LL}{\mathbb{L}} 							% left-derived
\newcommand{\RR}{\mathbb{R}} 							% typical Reedy category / right derived
\newcommand{\NN}{\mathbb{N}} 							% Natural numbers

\newcommand{\MM}{\mathbb{M}} 							% finite sets and monos
\newcommand{\MMSp}{\mathbb{M}\mhyphen\mathsf{spaces}} 		% M-spaces
\newcommand{\op}{\textup{op}} 							% op
\newcommand{\sSp}{\mathsf{sSpaces}}						% simplicial spaces
\newcommand{\GammaSp}{\Gamma\mhyphen\mathsf{spaces}}	% Gamma-spaces
\newcommand{\RRSp}{\RR\mhyphen\mathsf{spaces}}			% Gamma-spaces

\newcommand{\dS}{\mathsf{dSets}}							% dendroidal sets
\newcommand{\dSp}{\mathsf{dSpaces}}						% dendroidal spaces
\newcommand{\Sp}{\mathsf{sSets}}							% simplcial sets / spaces
\newcommand{\plusarrow}{\xrightarrow{\scriptscriptstyle+}}		% positive arrow
\newcommand{\minusarrow}{\xrightarrow{\scriptscriptstyle-}}		% negative arrow
\newcommand{\lminusarrow}{\xleftarrow{\scriptscriptstyle-}}		% negative arrow

\newcommand{\map}{\textup{map}}							% Map
\newcommand{\Hom}{\textup{Hom}}							% Hom

\newcommand{\colim}{\textup{colim}}						% colim

\newcommand{\hocolim}[1]{\underset{#1}{\textup{hocolim}}}						% hocolim

\newcommand{\sG}{\mathcal{G}} 							% curly G

\mathchardef\mhyphen="2D %  "math hyphen"

\newcommand{\pto}{\nrightarrow}%\rightharpoonup}
\newcommand{\pgets}{\nleftarrow}

\begin{document}

%% citations
%\nocite{*}

\title[Dendroidal spaces, $\Gamma$-spaces and the special BPQ theorem]{Dendroidal spaces, $\Gamma$-spaces and the special Barratt-Priddy-Quillen theorem}
\author{Pedro Boavida de Brito and Ieke Moerdijk}%
\address{Dept. of Mathematics, Instituto Superior Tecnico, Univ. of Lisbon, Av. Rovisco Pais, Lisboa, Portugal}%
\email{pedrobbrito@tecnico.ulisboa.pt}

\address{Department of Mathematics, Utrecht University, PO BOX 80.010,  3508 TA Utrecht, The Netherlands}%
\email{i.moerdijk@uu.nl}

%\subjclass{}
%\keywords{}%
%\date{\today}%
%\dedicatory{to ....}%
%\commby{}%
% ----------------------------------------------------------------
\begin{abstract}
We study the covariant model structure on dendroidal spaces, and establish direct relations to the homotopy theory of algebras over a simplicial operad as well as to the homotopy theory of special $\Gamma$-spaces. As an important tool in the latter comparison, we present a sharpening of the classical Barratt-Priddy-Quillen theorem.
\end{abstract}
\maketitle

%\setcounter{tocdepth}{1}
%\tableofcontents

% ----------------------------------------------------------------

\section{Introduction and statement of results}

The purpose of this paper is two-fold. The first is to provide a direct relation between
the theory of $\Gamma$-spaces and that of dendroidal spaces. The second is to 
provide a description, via dendroidal spaces and arguably simpler than the one 
in \cite{Heuts} and \cite{Lurie}, of the relation between algebras over operads and their $\infty$ 
counterparts.

\medskip
The category of $\Gamma$-spaces was introduced in the 1970s by Graeme Segal 
as a model for infinite loop spaces, or equivalently, for connective 
spectra. Later, this relation between connective spectra and 
$\Gamma$-spaces was cast in the more informative form of an equivalence of 
Quillen model categories by Bousfield and Friedlander \cite{Bousfield-Friedlander}. Dendroidal 
spaces, on the other hand, have been introduced and studied much more 
recently. They form a natural extension of simplicial spaces, and model 
topological operads in the same way as simplicial spaces model 
topological categories. More precisely, there is a so-called complete 
Segal model structure on simplicial spaces \cite{Rezk} which is Quillen equivalent 
to the Dwyer-Kan model structure on topological categories \cite{Bergner}, and 
this equivalence extends to an equivalence between a complete Segal 
model structure on dendroidal spaces and a similar Dwyer-Kan style model 
structure on simplicial operads \cite{Cisinski-Moerdijk3}. There are also ``discrete" Quillen 
equivalent versions of these model categories, namely the category of 
simplicial sets with the Joyal model structure whose fibrant objects 
are the quasi-categories of \cite{Joyal} and \cite{Lurie}, and the category of dendroidal 
sets with the operadic model structure of \cite{Cisinski-Moerdijk1} whose fibrant objects are 
the quasi-operads of \cite{Moerdijk-Weiss}.

\medskip
There is a natural contravariant functor from the category $\Omega$ of trees 
to the category of finite sets and partial maps, opposite to Segal's 
category $\Gamma$, which simply assigns to a tree T its set of leaves $\lambda(T)$. 
Our main theorem is that this functor $\lambda$ induces a Quillen 
equivalence between suitable model structures on dendroidal spaces and 
on $\Gamma$-spaces, respectively:

\begin{thm}\label{thm:one}
The functor $\lambda$ induces a Quillen equivalence between the 
category of dendroidal spaces and that of $\Gamma$-spaces, where the former 
is equipped with the covariant localisation of the complete Segal model 
structure, and the latter is equipped with the model structure for 
special $\Gamma$-spaces of \cite{Bousfield-Friedlander}. 
\end{thm}

As a consequence, a simple further left Bousfield localisation will then 
give a Quillen equivalence between a model structure on dendroidal 
spaces and the model structure on $\Gamma$-spaces whose fibrant objects are 
the \emph{very special} $\Gamma$-spaces of Bousfield and Friedlander, and hence 
a Quillen equivalence with the category of connective spectra. 

A key step in proving this theorem lies in identifying an explicit fibrant resolution 
of the discrete $\Gamma$-space $\Gamma(-,L)$ represented by a 
finite set $L$.  For the model structure corresponding to connective 
spectra, it is known \cite{Segal} that such a resolution is given by the 
$\Gamma$-space $B\Sigma^L$ corresponding to the symmetric monoidal category 
$\Sigma_L$ of finite sets labelled by L and bijections between them. As Segal explains,
this is essentially the content of the Barratt-Priddy-Quillen theorem. We 
will prove that the same is true \emph{unstably}, i.e. for the model structure for 
\emph{special} $\Gamma$-spaces. Our proof is partly based on ideas 
in \cite{Segal}, but does not use the relation to spectra. Instead, it is based on  a careful analysis of the interplay 
between the (generalised) Reedy and the projective model structures on 
$\Gamma$-spaces. In fact, we prove: 

\begin{thm}[Special Barratt-Priddy-Quillen theorem]\label{thm:introBPQ}
The canonical map of $\Gamma$-spaces $\Gamma(-, L) \to B\Sigma^L$ is a 
trivial cofibration in the Reedy model structure localised for special $\Gamma$-spaces. 
\end{thm}

The $\Gamma$-space $B\Sigma^L$ is not Reedy fibrant, so the theorem does not give a fibrant 
resolution for the Reedy model structure but, as a consequence of the 
theorem, it does for the projective model structure. Our proof gives a sharper result than the 
one for the \emph{very special} model structure, but we also believe it is perhaps more 
explicit and direct than other proofs in the literature. 

Another crucial  ingredient for the proof of Theorem \ref{thm:one} is a somewhat 
technical result which we refer to as a \emph{Yoneda lemma} for dendroidal 
spaces, since it is reminiscent of the Yoneda lemma for small 
categories. For a family of objects $\sigma: U \to \textup{ob}(C)$ in such a category $C$, 
this lemma can be interpreted as stating that the functor $U \to \sigma/C$ given by the 
identities on these objects gives a resolution of $U \to C$ as a cofibred category over $C$. 
We will construct for any dendroidal Segal space X and any family of objects 
of $X$, i.e. any map $\sigma : U \to X_\eta$ of simplicial sets, a new dendroidal 
space $\sigma/X$ and prove: 

\begin{thm}[Yoneda Lemma] This construction provides a weak equivalence 
\[
	\begin{tikzpicture}[descr/.style={fill=white}, baseline=(current bounding box.base)] ] 
	\matrix(m)[matrix of math nodes, row sep=1.5em, column sep=1.5em, 
	text height=1.5ex, text depth=0.25ex] 
	{
	\Omega[\eta] \otimes U & & \sigma/X \\
	& X & \\
	}; 
	\path[->,font=\scriptsize] 
		(m-1-1) edge node [auto] {$$} (m-1-3)
		(m-1-1) edge node [auto] {$$} (m-2-2)
		(m-1-3) edge node [auto] {$$} (m-2-2);
	\end{tikzpicture} 
\]
in the covariant model structure on dendroidal spaces over X, into a covariant fibration $\sigma/X \to X$.
\end{thm}

This covariant model structure is a relative version of the model 
structure featuring in Theorem \ref{thm:one}, and will be discussed in detail in Section \ref{sec:Yoneda}.

Using this dendroidal Yoneda lemma, we can also relatively easily deduce 
two other results. The first of these concerns the relation between 
algebras over a simplicial operad and dendroidal spaces. For this, we 
recall that for a (coloured) $\Sigma$-free simplicial operad $P$,  the 
category of $P$-algebras in simplicial sets carries a Quillen model 
structure transferred from the classical Kan-Quillen model structure on 
simplicial sets \cite{Berger-Moerdijk}. 

\begin{thm}\label{thm:intro-alg}
For any $\Sigma$-free simplicial operad $P$, there is a natural right Quillen equivalence from the model category of 
$P$-algebras to the category of dendroidal spaces over the nerve of $P$, when the latter is equipped with the covariant model structure. 
\end{thm}

This covariant model structure can in fact be defined on the category $\dSp{/X}$ of 
dendroidal spaces over any other dendroidal space $X$. 
For dendroidal spaces satisfying a Segal condition, as the nerves of simplicial operads 
do, it has the following invariance property: 

\begin{thm}[Invariance theorem]\label{thm:intro-inv}
Let $f : X \to Y$ be a map between 
dendroidal spaces satisfying the Segal condition. If this map is a Dwyer-Kan equivalence (i.e. a weak equivalence in the complete Segal model structure), then the induced Quillen pair 
\[
f_! : \dSp{/X} \leftrightarrows \dSp{/Y} : f^*
\]
yield a Quillen equivalence for the covariant model structures over $X$ and over $Y$, respectively. 
\end{thm}

Using the Quillen equivalence mentioned at the beginning between the complete Segal model structure 
on dendroidal spaces and the operadic model structure on dendroidal sets, we derive some results 
concerning dendroidal sets. First of all, this equivalence localises to an equivalence between 
the covariant model structures on dendroidal spaces  of Theorem \ref{thm:one} and the covariant model structure 
on dendroidal sets, and hence we obtain the following corollary: 

\begin{cor}\label{cor:dsets-gamma}
There is a zigzag of Quillen equivalences between the category of dendroidal 
sets equipped with the covariant model structure and category of $\Gamma$-spaces with the special model 
structure. 
\end{cor}

Next, the invariance theorem implies that for any dendroidal space $X$ 
satisfying the Segal condition, the covariant model structure on 
dendroidal spaces over $X$ is equivalent to the one over its completion. 
For a simplicial operad $P$, the completion of the nerve of $P$ is equivalent to
the homotopy coherent nerve $w^*(P)$ of $P$ \cite{Cisinski-Moerdijk3}. Hence, the invariance 
theorem will enable us to relate our comparison theorem to the work of Heuts, who was the first to study dendroidal covariant model structures and relate these to algebras over operads. In particular, we recover the following result, orginally proved by means of a model structure on marked 
dendroidal sets. 

\begin{cor}[of Theorems \ref{thm:intro-alg} and \ref{thm:intro-inv}, \cite{Heuts}]\label{cor:p-algs-dsets}
For any $\Sigma$-free simplicial operad $P$, there is a zig-zag of Quillen equivalences between the category of 
$P$-algebras and that of dendroidal sets over the homotopy coherent nerve of $P$. 
\end{cor}

A further localisation of the covariant model structure on dendroidal sets gives a model structure 
which is Quillen equivalent to that of very special $\Gamma$-spaces (Corollary \ref{cor:dsets-gamma}), or to group-like 
$E_\infty$-spaces (Corollary \ref{cor:p-algs-dsets}), and hence to infinite loop spaces. This 
model structure is studied under the name ``stable model structure" by Ba{\v{s}}i\'{c} and Nikolaus 
\cite{Basic-Nikolaus}. 

\medskip

The plan of our paper, then, is as follows: In Section \ref{sec:review}, we begin by 
reviewing the Reedy and projective model structures on diagrams of 
spaces indexed by so-called generalised Reedy categories, of which the categories $\Omega$ and $\Gamma$ are examples. In Section \ref{sec:dendroidalspaces}, we 
review the basic terminology concerning dendroidal spaces and the various 
model structures these carry. In Section \ref{sec:Yoneda}, we provide the necessary details 
concerning the covariant model structure and we prove our dendroidal 
Yoneda lemma. From this lemma, we easily derive in Section \ref{sec:invariance} the Invariance Theorem and in Section \ref{sec:algebras} our Theorem \ref{thm:intro-alg} about $P$-algebras. In Section \ref{sec:gamma-vs-dend}, we prove Theorem \ref{thm:one}. As mentioned above, the proof is based on the Yoneda lemma and our version of the Barratt-Priddy-Quillen theorem for \emph{special} $\Gamma$-spaces (Theorem \ref{thm:introBPQ}). The proof of the latter theorem is independent of the dendroidal theory, and can be read by itself. It only requires some notions from Section \ref{sec:review}. For this reason, we have written this proof in the form of an Appendix. 

\subsection*{Acknowledgements} We thank NWO for supporting several mutual visits. The first author was supported by FCT through grant SFRH/BPD/99841/2014.

\section{Review of generalised Reedy categories}\label{sec:review}

In this section we will review the definition of generalised Reedy categories and discuss the corresponding projective and Reedy model structures on their categories of presheaves.

\subsection{Generalised Reedy categories}  A generalised Reedy category - briefly Reedy category - is a small category $\RR$ equipped with a degree function $$d : \textup{objects}(\RR) \to \NN$$ and two classes of maps $\RR^-$ and $\RR^+$, both closed under composition. This structure is required to satisfy the following axioms: 
\begin{enumerate}
\item If $f : r \to s$ belongs to $\RR^+$ then $d(r) \leq d(s)$; and if it belongs to $\RR^-$ then $d(r) \geq d(s)$.
\item Every morphism $f$ in $\RR$ factors as $f = m e$ where $m \in \RR^+$ and $e \in \RR^-$, and this factorization is unique up to isomorphism.
\item A morphism belongs to both $\RR^+$ and $\RR^-$ if and only if it is an isomorphism. Also, if $f$ belongs to $\RR^-$ or $\RR^+$ and $f$ is degree-preserving then $f$ is an isomorphism.
\item For any isomorphism $\theta$ in $\RR$, if there is an $f \in \RR^+$ with $f \theta = f$ or a $g \in \RR^-$ with $\theta g = g$, then $\theta$ is the identity.
\end{enumerate}

This definition is that of a \emph{dualizable generalised} Reedy category of \cite{Berger-Moerdijk2}, but for brevity we will refer to these as \emph{Reedy categories}. We shall write $r \plusarrow s$ to indicate that the morphism belongs to $\RR^+$ and similarly for $\RR^-$. We call a morphism in $\RR^+$ \emph{positive}, and \emph{strictly positive} if it is not an isomorphism; similarly for \emph{negative} morphisms and $\RR^-$.

\subsection{The projective model structure} Let $\RR$ be any small category (not necessarily Reedy), and consider the category 
\[
\RRSp
\]
of simplicial presheaves on $\RR$; i.e. the category of contravariant functors from $\RR$ into simplicial sets. If $X$ is such a functor, we shall write $f^* : X_r \to X_s$ (or $f^* : X(r) \to X(s)$) for the action of a morphism $f : s \to r$ in $\RR$. This category comes equipped with a proper combinatorial model structure, in which a map $X \to Y$ is a weak equivalence (respectively, a fibration) if for each object $r$ in $\RR$ the map $X_r \to Y_r$ is a weak equivalence (respectively, a fibration) in the Kan-Quillen model structure on simplicial sets. (See for example \cite{Heller}.) 

\quad The category $\RRSp$ is tensored over simplicial sets. Indeed, if $X$ is a simplicial presheaf and $M$ is a simplicial set we can form a new simplicial presheaf $X \otimes M$ by $(X \otimes M)_r = X_r \times M$. Any object $r$ of $\RR$ defines a representable set-valued presheaf $\RR(-,r)$ which we view as an object of $\RRSp$ with discrete values (constant simplicial sets). With this notation in place, the generating cofibrations and trivial cofibrations are the morphisms of the form
\[
\RR(-,r) \otimes \partial \Delta[n] \to \RR(-,r) \otimes \Delta[n]
\]
(for $n\geq0$ and $0 \leq k \leq n$), and 
\[
\RR(-,r) \otimes \Lambda^k[n] \to \RR(-,r) \otimes \Delta[n]
\]
(for $n\geq1$ and $0 \leq k \leq n$), respectively.

\quad For later use, we also recall that $\RRSp$ is a category enriched in simplicial sets, with simplicial hom-sets defined for presheaves $X$ and $Y$ by
\[
\map (X,Y)_n = \Hom_{\RR} (X \otimes \Delta[n], Y) \; ,
\]
the right-hand side being the \emph{set} of natural transformations between presheaves on $\RR$.

\subsection{Resolutions and geometric realisation}
In the projective structure, the cofibrant objects are (retracts of) coproducts of objects of the form $\RR(-,r) \otimes K$. As is well known, each object $X$ has a canonical simplicial resolution $\mathbb{X}_\bullet$ by such cofibrant objects, where
\[
\mathbb{X}_n = \coprod_{r_0 \to \dots \to r_n} \RR(-,r_0) \otimes X_{r_n} \; .
\]
This means that $X$ is weakly equivalent to the geometric realization of this simplicial object. If we express $X_{r_n}$ itself as the realisation of a simplicial space which in each degree is a sum of standard simplices, then we obtain another canonical simplicial resolution $\mathbb{Y}_\bullet$ where each $\mathbb{Y}_n$ is a coproduct of objects of the form $\RR(-,r) \otimes \Delta^n$.

Notice that a simplicial object in $\RR$-spaces is a bisimplicial $\RR$-set and its realisation agrees with its diagonal. In particular, there is no need to derive it. In more general contexts, by the phrase \emph{geometric realisation} we will always mean the homotopy colimit (over $\Delta)$.

We will use this type of resolution and properties of geometric realisation in a general form that we now state but which will only be used much later, in the proofs in sections \ref{sec:invariance}, \ref{sec:algebras}, and \ref{sec:gamma-vs-dend}.

Let
\[
g_! : \mathbf{F} \leftrightarrows \mathbf{E} : g^*
\]
be a Quillen pair, which restricts to a Quillen pair
\[
\widetilde{g}_! : \mathbf{F}_{loc} \leftrightarrows \mathbf{E}_{loc} : \widetilde{g}^*
\]
of given left Bousfield localisations.  In the cases of interest to us, $\mathbf{F}$ will be the category of dendroidal spaces with the projective model structure (or a slice of it).

\begin{lem}\label{lem:reduction}
In the setting above, suppose that
\begin{enumerate}
\item[(a)] $\RR g^*$ commutes with geometric realisation, and
\item[(b)] the geometric realisation of a degreewise local simplicial object in $\mathbf{E}$ is local.
\end{enumerate}
If $W_\bullet$ is a simplicial object in $\mathbf{F}$ for which the derived unit map
$
W_n \to \RR \widetilde{g}^* \LL \widetilde{g}_! W_n
$
is a local weak equivalence for each $n$, then 
$
|W_\bullet| \to \RR \widetilde{g}^* \LL \widetilde{g}_! |W_{\bullet}|
$
is a local weak equivalence as well.
\end{lem}
\begin{proof}
Let us first observe that $\RR \widetilde{g}^*$ also commutes with realisation. Indeed, given a simplicial object $E_\bullet$  in $\mathbf{E}$ and a degreewise weak equivalence $E_\bullet \to \widetilde{E}_{\bullet}$ into a degreewise local object, we have that
\[
| \RR \widetilde{g}^*(E_\bullet) | \simeq 
|\RR g^* (\widetilde{E}_\bullet) | \simeq
\RR g^* |(\widetilde{E}_\bullet) | \simeq
\RR \widetilde{g}^* | E_{\bullet} |
\]
where the second weak equivalence holds by assumption $(a)$ and the third by assumption $(b)$. 

Thus, in the square
\begin{equation*}%\label{eq:squnder}
	\begin{tikzpicture}[descr/.style={fill=white}, baseline=(current bounding box.base)] ] 
	\matrix(m)[matrix of math nodes, row sep=2.5em, column sep=2.5em, 
	text height=1.5ex, text depth=0.25ex] 
	{
	\vert W_{\bullet} \vert & \RR \widetilde{g}^* \LL \widetilde{g}_! \vert W_{\bullet}\vert  \\
	\vert \RR \widetilde{g}^* \LL \widetilde{g}_! W_{\bullet} \vert & \RR \widetilde{g}^* \vert \LL \widetilde{g}_! W_{\bullet} \vert  \\
	}; 
	\path[->,font=\scriptsize] 
		(m-1-1) edge node [auto] {$$} (m-1-2)
		(m-2-1) edge node [auto] {$$} (m-2-2)
		(m-2-1) edge node [left] {$$} (m-1-1)
		(m-2-2) edge node [auto] {$$} (m-1-2);
	\end{tikzpicture} 
\end{equation*}
the lower map is a weak equivalence by the previous paragraph. The result now follows since the vertical maps are weak equivalences (the left-hand one is so by assumption).
\end{proof}

\subsection{The Reedy model structure}\label{sec:reedyref}
Now let $\RR$ be a (generalised) Reedy category in the sense above. Then the category $\RRSp$ carries another (proper, combinatorial) model structure with weak equivalences given objectwise, as in the projective model structure. The fibrations and cofibrations are defined explicitly in terms of latching and matching objects $deg(X)(r)$\footnote{$deg$ stands for degeneracies, reflecting the examples we work with.} and $X(\partial r)$ associated to a simplicial presheaf $X$ and an object $r$ in $\RR$,
\[
deg(X)(r) = \underset{r \minusarrow s}{\colim} X_s
\quad \textup{and} \quad 
X(\partial r) = \lim_{s \plusarrow r} X_s
\]
where the colimit ranges over the full subcategory $(r/\RR)^{-}$ of $r/\RR$ whose objects are \emph{properly negative} morphisms out of $r$, and the limit ranges over the full subcategory $(\RR/r)^{+}$ of properly positive morphisms into $\RR$.

\quad For later use, we observe that the category $(r/\RR)^{-}$ is obviously a subcategory of the larger full subcategory of $r/\RR$ whose objects are those morphisms $f : r \to s$ which factor through a properly negative map $g : r \minusarrow r^\prime$. Moreover, $(r/\RR)^{-}$ is cofinal in this larger subcategory, so we might as well have defined $deg(X)(r)$ as the colimit over this larger category. In particular, $deg(X)(r)$ is functorial in $r$. A similar remark applies to $(\RR/r)^+$ and the definition of $X(\partial r)$. With this functoriality in $r$, observe that the obvious maps
\[
deg(X)(r) \to X(r) \to X(\partial r)
\]
are natural in $r$.

\quad We can now define the fibrations and cofibrations of the Reedy model structure: a morphism $X \to Y$ is a (Reedy) cofibration if for each object $r$ in $\RR$, the map
\[
deg(Y)(r) \cup_{deg(X)(r)} X_r \to Y_r
\]
is a cofibration in the projective model structure on $Aut(r)\mhyphen\mathsf{spaces}$, the category of spaces with a (right) action of the group $Aut(r)$; this means that this map is a monomorphism of simplicial sets with a free $Aut(r)$-action on the complement of its image. And a morphism $Y \to X$ is a (Reedy) fibration if for each object $r$ in $\RR$,  the map
\[
Y_r \to X_r \times_{X(\partial r)} Y(\partial r)
\]
is a Kan fibration of simplicial sets.

\quad For an object $r$ in $\RR$, let us write $\partial \RR(-,r)$ for the subobject of the representable presheaf $\RR(-,r)$ given by those morphisms $s \to r$ which factor through \emph{some} properly positive morphism $r^\prime \plusarrow r$. Thus
\[
X(\partial r) = \map_{}(\partial \RR(-,r), X)
\]
With this notation, the generating (trivial) cofibrations of the Reedy model structure are the maps of the form
\[
\partial \RR(-,r) \otimes B \cup_{\partial \RR(-,r) \otimes A} \RR(-,r) \otimes A \to \RR(-,r) \otimes B
\]
where $A \hookrightarrow B$ is a generating (trivial) cofibration of simplicial sets.

\quad By definition, an object $X$ is Reedy fibrant if and only if, for each $r$, the map $X(r) \to X(\partial r)$ is a Kan fibration. More generally, let $V$ be a subobject of $\partial \RR(-,r) \subset \RR(-,r)$. That is, $V$ corresponds to a set of morphisms with target $r$ which factor through some properly positive morphism and are closed under precomposition with morphisms in $\RR$. Write
\[
X(\partial_V r) = \map_{}(V, X) = \lim_{t \to r \in V} X(t)
\]

\begin{lem}\label{lem:Vreedy}
If $X$ is Reedy fibrant, then for any such $V$ the map $X(\partial r) \to X(\partial_V r)$ is a Kan fibration.
\end{lem}
\begin{proof}
For simplicity, we assume that $\RR$ is countable as this is the case in all our examples.
The proof is by induction on the degree $d(r)$ of $r$. The case $d(r) = 0$ is clear. Suppose the statement is true for all $s$ with $d(s) < d(r)$ and all $V$. Let $\{s_n \plusarrow r\}_{n \geq 0}$ be a sequence of properly positive morphisms such that $s_n \to r$ does not factor through any $s_i \to r$ with $i < n$ but any properly positive $t \to r$ factors through \emph{some} $s_n \to r$.

Let
\[
P_n = \lim X(s)
\]
where the limit ranges over all $s \to r$ which are either in $V$ or factor through some $s_k \to r$ with $k \leq n$. This gives a tower
\[
P_{-1} \gets P_{0} \gets P_{1} \gets \dots 
\]
whose inverse limit is exactly the map $X(\partial r) \to X(\partial_V r)$ in the statement of the lemma. Therefore it is enough to prove that each $P_n \to P_{n-1}$ is a fibration. But $P_n \to P_{n-1}$ is the pullback of 
\[
X(s_n) \to \lim_{(t \plusarrow s_n) \in W} X(t)
\]
where the limit is taken over the set $W$ of morphisms $t \plusarrow s_n$ which either factor through some $s_k \to s_n$ with $k < n$ or are in $V$. This is a fibration by induction since $d(s_n) < d(r)$. This completes the proof.
\end{proof}

\begin{rem}\label{rem:RfibV}
In exactly the same way, one proves that if $Y \to X$ is a generating Reedy fibration then for any subobject $V$ of $\partial \RR(-,r)$ the maps
\[
Y(r) \to X(r) \times_{X(\partial_V r)} Y(\partial_V r) \quad \mbox{ and } \quad Y(\partial r) \to X(\partial r) \times_{X(\partial_V r)} Y(\partial_V r)
\]
are Kan fibrations.
\end{rem}

Next, let us take a closer look at cofibrations and cofibrant objects in the Reedy model structure. The examples below all have the property that every negative morphism in the Reedy category $\RR$ has a section, and that every diagram of the form $s \lminusarrow r \minusarrow t$ has an absolute pushout. This implies that for every simplicial presheaf $X$ and every object $r$ in $\RR$, the map $deg(X)(r) \to X(r)$ is automatically injective, and that for a morphism $X \to Y$ between presheaves, the map $deg Y(r) \cup_{deg X(r)} X(r) \to Y(r)$ is injective for each $r$ if and only if $X \to Y$ is itself (objectwise) injective. Thus, a map $X \to Y$ is a Reedy cofibration if and only if it is a monomorphism with the property that, for each $r \in \RR$, the group $Aut(r)$ acts freely on the complement of $deg(Y)(r) \cup_{deg(X)(r)} X(r) \to Y(r)$. We call such monomorphisms $X \to Y$ \emph{normal} monomorphisms, and call an object $X$ \emph{normal} if $\varnothing \to X$ is a normal monomorphism, i.e. if $Aut(r)$ acts freely on the complement of $deg(X)(r) \hookrightarrow X(r)$ for each $r$. For the record, we summarize this observation as follows.

\begin{lem}\label{lem:reedy-cof}
Let $\RR$ be a generalised Reedy category in which every negative morphism has a section and in which every diagram of the form $\cdot \gets \cdot \to \cdot$ of negative morphisms has an absolute pushout. Then the Reedy cofibrations are the normal monomorphisms, and the cofibrant objects are the normal objects.
\end{lem}

\begin{rem}
Remark \ref{rem:RfibV} implies, by taking $V = \varnothing$, that every Reedy fibration is an objectwise (projective) Kan fibration. Since the projective and Reedy model structures have the same weak equivalences, this means that the identity functors form a Quillen equivalence
\[
(\RRSp)_P \leftrightarrows (\RRSp)_R
\]
between these model structures ($P$ for projective, $R$ for Reedy and left adjoint goes from left to right).
\end{rem}

\section{Dendroidal spaces}\label{sec:dendroidalspaces}

The category $\Omega$ of trees is the category indexing \emph{dendroidal} sets and spaces. Its objects are non-empty finite rooted trees, possibly with leaves, as in 
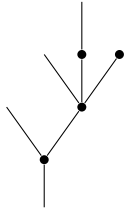
\[
\begin{tikzpicture} 
[level distance=7mm, 
every node/.style={fill, circle, minimum size=.1cm, inner sep=1.2pt}, 
level 1/.style={sibling distance=15mm}, 
level 2/.style={sibling distance=10mm}, 
level 3/.style={sibling distance=5mm}]

%centre tree
\node (centretree)[style={color=white}] {} [grow'=up] 
child {node (r) {} 
	child{
	}
	child{ node {}
		child
		child{ node (centrev){} child }
		child{ node {} }
	}
};
\end{tikzpicture}
%}
\]

%\medskip

Each such tree generates a coloured symmetric operad $\Omega(T)$ whose colours are the edges of the tree and whose operations are generated by the vertices. The morphisms in $\Omega$ from a tree $S$ to a tree $T$ are precisely the operad maps $\Omega(S) \to \Omega(T)$. Any such morphism is a composition of elementary morphisms: besides the isomorphisms between trees, these are degeneracies, internal and external faces, much as for $\Delta$, and illustrated in the following picture.

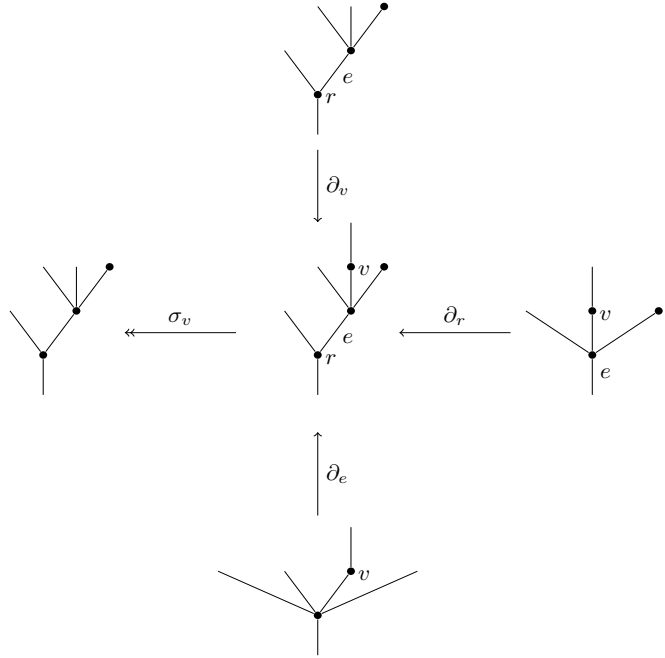
\[
\resizebox{250pt}{250pt}{% horizontal, vertical
\begin{tikzpicture} 
[level distance=7mm, 
every node/.style={fill, circle, minimum size=.1cm, inner sep=1.2pt}, 
level 1/.style={sibling distance=15mm}, 
level 2/.style={sibling distance=10mm}, 
level 3/.style={sibling distance=5mm}]

%centre tree
\node (centretree)[style={color=white}] {} [grow'=up] 
child {node (r) {} 
	child{
	}
	child{ node {}
		child
		child{ node (centrev){} child }
		child{ node {} }
	}
};

%tree above
\node(treeabove)[style={color=white}, above=4cm of centretree] {} [grow'=up] 
child {node (abover) {} 
	child
	child{ node {}
		child
		child
		child{ node {} }
	}
};

%tree to the left
\node(lefttree)[style={color=white}, left=4cm of centretree] {} [grow'=up] 
child {node (leftr) {} 
	child{
	}
	child{ node {}
		child
		child{ }
		child{ node {} }
	}
};

%tree to the right
\node(righttree)[style={color=white}, right=4cm of centretree] {} [grow'=up] 
child {node(rightu) {} 
	child{
	}
	child{ node (w){}
		child
	}
	child{ node {}
	}
};

%tree below
\node(treebelow)[style={color=white}, below=4cm of centretree] {} [grow'=up] 
child {node {} 
	child
	child
	child{ node(belowv) {}
		child
	}
	child
};

\tikzstyle{every node}=[]

%arrows
\draw[->] ($(treeabove) + (0,-5pt)$) -- node[right]{$\partial_v$} ($(centretree) + (0, 80pt)$);
\draw[->] ($(treebelow) + (0,65pt)$) -- node[right]{$\partial_e$} ($(centretree) + (0, -15pt)$);
\draw[<<-] ($(lefttree) + (35pt,30pt)$) -- node[above]{$\sigma_v$} ($(centretree) + (-35pt, 30pt)$);
\draw[->] ($(righttree) + (-35pt,30pt)$) -- node[above]{$\partial_r$} ($(centretree) + (35pt, 30pt)$);

%labels
%centre tree
\node at ($(r) + (-17:6pt)$) {$r$};
\node at ($(r) + (30:15pt)$) {$e$};
\node at ($(centrev) + (-17:6pt)$)  {$v$};

%upper tree
\node at ($(abover) + (-17:6pt)$) {$r$};
\node at ($(abover) + (30:15pt)$) {$e$};

%lower tree
\node at ($(belowv) + (-17:6pt)$) {$v$};

%right tree
\node at ($(rightu) + (-55:10pt)$) {$e$};
\node at ($(w) + (-17:6pt)$) {$v$};

\end{tikzpicture} 
}
\]
In the picture, $\partial_v$ is an external (top) face, $\partial_r$ an external (root) face, $\partial_e$ an inner face and $\sigma_v$ a degeneracy map.

Among all objects of $\Omega$, some have a special status and deserve their own notation: $\eta$ is the tree consisting of a single edge; for a non-negative integer $n$, the $n$-corolla $C_n$ is the tree with a single vertex and $n+1$ edges.
%\medskip

\quad The category $\dS$ of \emph{dendroidal sets} is the category of contravariant functors from $\Omega$ into sets, while the category $\dSp$ is that of \emph{dendroidal spaces}, i.e. contravariant functors from $\Omega$ into simplicial sets. The action of a morphism $\alpha : S \to T$ on a dendroidal set or space $X$ is denoted $\alpha^* : X_T \to X_S$. 
The dendroidal set represented by a tree $T$ is denoted $\Omega[T]$. The category $\Omega$ and the associated categories of dendroidal sets and spaces are discussed in detail in \cite{Moerdijk-Weiss}, \cite{Cisinski-Moerdijk1} and \cite{Cisinski-Moerdijk2}.

The simplicial category $\Delta$ embeds in $\Omega$, where we regard $[n]$ as a linear tree with $n$ vertices. We denote this inclusion by $i$.

\medskip
In the definition below, and unless stated otherwise, we write $\RR \map$ for the homotopy-invariant (alias derived) mapping space in the category of dendroidal spaces with respect to levelwise weak equivalences.

\begin{defn}\label{defn:Segaldend}
A dendroidal space $X$ is called a \emph{Segal} dendroidal space if the map
\[
%X(T) = 
\RR \map(\Omega[T], X) \to \RR \map(Sc[T], X)
\]
is a weak equivalence of spaces for each tree $T$. Here $Sc[T]$ is the Segal core of $T$, i.e. the union of all the corollas of $T$. (This union is indexed over the vertices $v$ of $T$.)
\end{defn}

\begin{rem} These Segal cores are discussed in detail in \cite{Cisinski-Moerdijk2}. In particular, it follows from Proposition 2.5 there that $X$ is a Segal dendroidal space if and only if the map
\[
%X(T) = 
\RR \map(\Omega[T], X) \to \RR \map(\Lambda^e[T], X)
\]
is a weak equivalence of spaces for each tree $T$ and each \emph{inner} edge $e$ in $T$. Here $\Lambda^e[T] \subset \partial \Omega[T]$ is the union of all faces except the one given by contracting $e$.
\end{rem}

Let $E$ denote the simplicial set given as the nerve of the groupoid with two objects $x$ and $y$ and exactly two non-identity morphisms $x \to y$ and $y \to x$. Recall from \cite{Rezk} that a simplicial space is \emph{complete} if the evaluation at one of the points $\RR\map(E,X) \to \RR\map(\Delta[0], X)$ is a weak equivalence. The following definition is equivalent to the one given in \cite{Cisinski-Moerdijk2}.

\begin{defn}\label{defn:completedend}
A Segal dendroidal space $X$ is \emph{complete} if the underlying Segal space $i^* X$ is complete.
\end{defn}

\subsection{Homotopy theories} The tree category $\Omega$ is a generalised Reedy category: the degree function $d$ is defined by $d(T) = $ number of vertices in $T$, while $\Omega^+$ is generated by faces and isomorphisms and $\Omega^-$ by degeneracies and isomorphisms. Therefore, the category of dendroidal spaces has a Reedy model structure. A dendroidal space $X$ is Reedy fibrant if the map
\[
X_T = \map(\Omega[T], X) \to \map(\partial \Omega[T], X)
\]
is a Kan fibration for each tree $T$, where $\Omega[T]$ is the representable dendroidal set (viewed as a dendroidal space) and $\partial \Omega[T]$ is its boundary. It is Reedy cofibrant if $Aut(T)$ acts freely on $X(T) \backslash deg(X)(T)$, and this happens to be equivalent to acting freely on all of $X(T)$.

There is a diagram of model categories and left Quillen functors:
\begin{equation}\label{eq:omega}
	\begin{tikzpicture}[descr/.style={fill=white}, baseline=(current bounding box.base)] ] 
	\matrix(m)[matrix of math nodes, row sep=2.5em, column sep=2.5em, 
	text height=1.5ex, text depth=0.25ex] 
	{
	\dSp_{P} & \dSp_{PS} & \dSp_{PSC} \\
	\dSp_{R} & \dSp_{RS} & \dSp_{RSC} \\
	}; 
	\path[->,font=\scriptsize] 
		(m-1-1) edge node [auto] {$$} (m-1-2);
	\path[->,font=\scriptsize] 
		(m-2-1) edge node [auto] {$$} (m-2-2);
	\path[->,font=\scriptsize] 
		(m-1-1) edge node [auto] {$\simeq$} (m-2-1);
	\path[->,font=\scriptsize] 		
		(m-1-2) edge node [auto] {$\simeq$} (m-2-2);
	\path[->,font=\scriptsize] 
		(m-1-2) edge node [auto] {} (m-1-3);
	\path[->,font=\scriptsize] 		
		(m-2-2) edge node [auto] {} (m-2-3);
	\path[->,font=\scriptsize] 		
		(m-1-3) edge node [auto] {$\simeq$} (m-2-3);
	\end{tikzpicture} 
\end{equation}
where $P$ and $R$ stand for projective and Reedy, $S$ stands for Segal and $C$ for complete. The fibrant objects in $\dSp_{PS}$ are the Segal dendroidal spaces which are projectively fibrant, whereas the fibrant objects in $\dSp_{RS}$ are the Segal dendroidal spaces which are Reedy fibrant. Similarly for $PSC$ and $RSC$. From left to right, the model structures on each column are obtained via left Bousfield localisation.

\begin{rem}\label{rem:homotopical}
Definitions \ref{defn:Segaldend} and \ref{defn:completedend} are in agreement with \cite{Cisinski-Moerdijk2}.
Indeed, if $X$ is a \emph{Reedy fibrant} dendroidal space, then the maps in those definitions are Reedy fibrations, and all the derived mapping spaces are identified with their non-derived versions. This is the case because $Sc[T] \to \Omega[T]$, $\Lambda^e[T] \to \Omega[T]$ and $i_!(E) \to i_!(\Delta^0)$ are inclusions into normal objects and hence Reedy cofibrations.
\end{rem}

\begin{defn}
A map $f : X \to Y$ between Segal dendroidal spaces is a Dwyer-Kan equivalence if $i^*f$ is essentially surjective (as a map of Segal spaces) and $f$ is fully faithful in the sense that the square
\begin{equation}\label{eq:ff}
	\begin{tikzpicture}[descr/.style={fill=white}, baseline=(current bounding box.base)] ] 
	\matrix(m)[matrix of math nodes, row sep=2.5em, column sep=2.5em, 
	text height=1.5ex, text depth=0.25ex] 
	{
	X_{C_n} & X_{\eta}^{\times n+1} \\
	Y_{C_n} & Y_{\eta}^{\times n+1} \\
	}; 
	\path[->,font=\scriptsize] 
		(m-1-1) edge node [auto] {$$} (m-1-2);
	\path[->,font=\scriptsize] 
		(m-2-1) edge node [auto] {$$} (m-2-2);
	\path[->,font=\scriptsize] 
		(m-1-1) edge node [auto] {$$} (m-2-1);
	\path[->,font=\scriptsize] 		
		(m-1-2) edge node [auto] {$$} (m-2-2);
	\end{tikzpicture} 
\end{equation}
is homotopy cartesian, for each $n$. Here the horizontal maps are induced by the inclusion $\amalg_{n+1} \eta \to C_n$ which selects the $n$ leaves and the root of the corolla. 
\end{defn}

In the complete model structures $\dSp_{RSC}$ and $\dSp_{PSC}$, the weak equivalences between Segal dendroidal spaces (which are not necessarily complete as such) are precisely the Dwyer-Kan equivalences \cite{Rezk}. Furthermore, many fibrations in these model structures can be identified:
\begin{prop}\label{prop:fibcomp}
Suppose $B$ is a projective (or Reedy) fibrant Segal dendroidal space, but not necessarily complete. A map $f : X \to B$ is a projective (or Reedy) complete fibration if and only if it satisfies the following conditions:
\begin{enumerate}
\item[$(i)$] $X$ is a Segal dendroidal space and $f$ is a projective (or Reedy) fibration;
\item[$(ii)$] $i^* f : i^*X \to i^*B$ is a fibrewise complete Segal space, that is, the map
\[
\RR \map(E, i^*X) \to \RR \map(E,i^*B) \times^h_{\RR \map(\Delta[0], i^*B)} \RR \map(\Delta[0], i^*X)
\]
is a weak equivalence. In other words, the square
\begin{equation*}%\label{eq:squnder}
	\begin{tikzpicture}[descr/.style={fill=white}, baseline=(current bounding box.base)] ] 
	\matrix(m)[matrix of math nodes, row sep=2.5em, column sep=2.5em, 
	text height=1.5ex, text depth=0.25ex] 
	{
	(i^*X)^{he}_1 & (i^*X)_0 \\
	(i^*B)^{he}_1 & (i^*B)_0 \\
	}; 
	\path[->,font=\scriptsize] 
		(m-1-1) edge node [auto] {$$} (m-1-2)
		(m-2-1) edge node [auto] {$$} (m-2-2)
		(m-1-1) edge node [left] {$$} (m-2-1)
		(m-1-2) edge node [auto] {$$} (m-2-2);
	\end{tikzpicture} 
\end{equation*}
is homotopy cartesian, where $(Z_1)^{he}$ denote the space of homotopy invertible morphisms of a Segal space $Z$.
\end{enumerate}
\end{prop}

\begin{rem}
As before, in the Reedy case, point $(ii)$ is equivalent to the statement that the non-derived map
\[
\map(E, i^*X) \to \map(E,i^*B) \times_{\map(\Delta[0], i^*B)} \map(\Delta[0], i^*X)
\]
is a trivial Reedy fibration. 
\end{rem}

\begin{proof}
We prove this for the Reedy model structure, the projective case is identical. 

One implication is immediate: if $f$ is a fibration in $\dSp_{RSC}$ then $(i)$ holds since $f$ is in particular a fibration in $\dSp_{RS}$, and $(ii)$ holds since the inclusion $i_!\Delta[0] \to i_!E$ is a trivial cofibration by definition of the complete model structure (and so $f$, being a fibration, has the right lifting property with respect to it).

\quad For the reverse implication, factor $f$ as a trivial cofibration $j$ followed by a fibration $g$ in $\dSp_{RSC}$
\[
X \overset{\sim}{\hookrightarrow} Y \to B \; .
\]
Since $g$ is a fibration in $\dSp_{RSC}$, it is also a fibration in $\dSp_{RS}$, and so it follows from the assumption that $B$ is a Segal dendroidal space that $Y$ is also a Segal dendroidal space. Weak equivalences between dendroidal Segal spaces in $\dSp_{RSC}$ are precisely the Dwyer-Kan equivalences, therefore $j$ is a Dwyer-Kan equivalence. The claim is that $j$ is in fact an objectwise weak equivalence. This is enough to conclude that $f$ is a fibration in $\dSp_{RSC}$ since in any left Bousfield localisation $L_S M$ of a model category $M$ at a set of maps $S$, if a map $f : X \to B$ is a fibration in $M$, $g : Y \to B$ a fibration in $L_S M$, and $j : X \to Y$  satisfying $f = gj$ is a weak equivalence in $M$, then $f$ is a fibration in $L_S M$.

We deduce the claim from the corresponding statement for Segal spaces. Namely, applying $i^*$ we obtain maps of Segal spaces 
\[
i^*X \to i^*Y \to i^*B
\]
where the composite $i^*f$ and the right-hand map $i^*g$ are fibrewise complete Segal spaces, and $i^*j$ is a Dwyer-Kan equivalence. It follows from \cite[Proposition B.8]{Boavida-Weiss} that $i^*j$ is an objectwise weak equivalence. In particular $j(\eta) : X_{\eta} \to Y_{\eta}$ is a weak equivalence of spaces. Together with the assumption that $j$ is fully faithful, this implies that $j$ is an objectwise weak equivalence.
\end{proof}

\section{Covariant fibrations and the Yoneda lemma for dendroidal spaces}\label{sec:Yoneda}

After some important definitions, we state and prove the main technical lemma of the paper which we call the Yoneda lemma.

\subsection{Covariant fibrations}
We begin with the definition of a covariant fibration (also called left fibration). In line with the previous section (cf. Remark \ref{rem:homotopical}), we first give a formulation in homotopical terms, i.e. only depending on the levelwise weak equivalences of dendroidal spaces and irrespective of a particular model structure.

Throughout this section, $B$ will denote a fixed Segal dendroidal space.

\begin{defn}
A map $X \to B$ of dendroidal spaces is a \emph{covariant fibration} if for each tree $T$ the diagram of simplicial sets
\begin{equation}\label{eq:covfib}
	\begin{tikzpicture}[descr/.style={fill=white}, baseline=(current bounding box.base)] ] 
	\matrix(m)[matrix of math nodes, row sep=2.5em, column sep=2.5em, 
	text height=1.5ex, text depth=0.25ex] 
	{
	X_{T} & \map({\lambda(T)}, X_\eta) \\
	B_{T} & \map({\lambda(T)}, B_\eta) \\
	}; 
	\path[->,font=\scriptsize] 
		(m-1-1) edge node [auto] {$$} (m-1-2);
	\path[->,font=\scriptsize] 
		(m-2-1) edge node [auto] {$$} (m-2-2);
	\path[->,font=\scriptsize] 
		(m-1-1) edge node [auto] {$$} (m-2-1);
	\path[->,font=\scriptsize] 		
		(m-1-2) edge node [auto] {$$} (m-2-2);
	\end{tikzpicture} 
\end{equation}
is homotopy cartesian, where $\lambda(T)$ denotes the set of leaves of $T$.
\end{defn}

Since $B$ is a Segal dendroidal space, we see by writing $B_T$ as a homotopy limit of corollas that it is enough to require that $X$ is a Segal dendroidal space and that the square (\ref{eq:covfib}) is homotopy cartesian when $T$ is a corolla.

\begin{rem}
A map $X \to B$ is a covariant fibration if and only if the map
\[
\RR \map(\Omega[T], X) \to \RR \map(\Lambda^e[T], X) \times^h_{\RR \map(\Lambda^e[T], B)} \RR \map(\Omega[T], B) 
\]
for every horn inclusion $\Lambda^e[T] \subset \Omega[T]$ with $e$ an inner edge or a top vertex, is a weak equivalence.
\end{rem}

We will add the adjective \emph{Reedy} or \emph{projective} whenever we require a covariant fibration to be a fibration in one of these two senses.

\subsection{The covariant model structure} Let $B$ be a projectively fibrant Segal dendroidal space. The projective \emph{covariant model structure} (alias left fibration model structure) on $\dSp$ is the left Bousfield localisation of $${(\dSp_{PS})}{/B} \; ,$$ the overcategory of the projective Segal model structure over $B$, at the maps selecting the leaves
\begin{equation}\label{eq:tophorn}
\coprod_{\lambda(T)} \eta \hookrightarrow \Omega[T] \xrightarrow{x} B
\end{equation}
for any tree $T$ and any $x \in B(T)$. Its fibrant objects are precisely the covariant projective fibrations.

These projective and Reedy model structures will be denoted 
$$({\dSp{/B}})_{P,cov} \quad \mbox{ and } \quad ({\dSp{/B}})_{R,cov} \; ,$$
respectively. 
As for any left Bousfield localisation, a map $f : X \to Y$ of dendroidal spaces over $B$ is a projective (or Reedy) \emph{covariant weak equivalence} if for every projective (or Reedy) covariant fibration $Z \to B$, the induced map
\[
f^* : \RR \map_B(Y, Z) \to \RR \map_B(X, Z) 
\]
is a weak equivalence.

In either setting, we have the following straightforward characterization of weak equivalences:
\begin{lem}\label{lem:covwe}
A map $f : X \to Y$ between of covariant fibrations over $B$ is a covariant weak equivalence if and only if the map on colours $X_\eta \to Y_\eta$ (over $B_\eta$) is a weak equivalence.
\end{lem}
\begin{proof}
A map between fibrant objects in the localised setting (i.e. covariant fibrations over $B$) is a weak equivalence if and only if it is a levelwise weak equivalence. And from the definitions we have that a map between covariant fibrations over $B$ is a weak equivalence at every tree $T$ if and only if it is a weak equivalence at $\eta$, which proves the lemma.
\end{proof}

The projective and Reedy covariant model structures are in fact localisations of the corresponding \emph{complete} model structures by virtue of the following Proposition.

\begin{prop}\label{prop:covfib_is_compl}
If $f : X \to B$ is a covariant fibration, then it is a complete fibration. In particular, if $B$ is a complete dendroidal Segal space, then so is $X$. 
\end{prop}

\begin{proof}
If $f$ is a covariant fibration of dendroidal Segal spaces then $i^*f$ is a covariant fibration of Segal spaces. But covariant fibrations of Segal spaces are fibrewise complete by \cite{Boavida}. Therefore, using Proposition \ref{prop:fibcomp}, we conclude that $f$ is a fibration in the complete Segal dendroidal space model structure.
\end{proof}

\begin{cor}
If $X \to Y$ is a map of Segal dendroidal spaces over $B$ which is a Dwyer-Kan equivalence, then it is a covariant weak equivalence.
\end{cor}

\subsection{The stable model structure} 
When $B$ is the terminal object, a further localisation of $$\dSp_{R,cov} = {(\dSp{/*})}_{R,cov}$$ 
at the maps
\[
\coprod_{i = 1}^k \eta \hookrightarrow C_{k+1} 
\]
for $k \geq 1$, selecting the root and all leaves except one, defines the \emph{stable model structure} \cite{Basic-Nikolaus}. (It is in fact enough to localise at a single map given by $k = 2$.)

We denote this model category by $\dSp_{R,stable}$ and the corresponding projective version by $\dSp_{P,stable}$.

\medskip
In summary, the ladder (\ref{eq:omega}) of localisations and Quillen equivalences can be prolonged as
\begin{equation}\label{eq:plong}
	\begin{tikzpicture}[descr/.style={fill=white}, baseline=(current bounding box.base)] ] 
	\matrix(m)[matrix of math nodes, row sep=2.5em, column sep=2.5em, 
	text height=1.5ex, text depth=0.25ex] 
	{
	\dSp_{PSC} & \dSp_{P,cov} & \dSp_{P,stable} \\
	\dSp_{RSC} & \dSp_{R,cov} & \dSp_{R,stable} \\
	}; 
	\path[->,font=\scriptsize] 
		(m-1-1) edge node [auto] {$$} (m-1-2);
	\path[->,font=\scriptsize] 
		(m-2-1) edge node [auto] {$$} (m-2-2);
	\path[->,font=\scriptsize] 
		(m-1-1) edge node [auto] {$\simeq$} (m-2-1);
	\path[->,font=\scriptsize] 		
		(m-1-2) edge node [auto] {$\simeq$} (m-2-2);
	\path[->,font=\scriptsize] 
		(m-1-2) edge node [auto] {} (m-1-3);
	\path[->,font=\scriptsize] 		
		(m-2-2) edge node [auto] {} (m-2-3);
	\path[->,font=\scriptsize] 		
		(m-1-3) edge node [auto] {$\simeq$} (m-2-3);
	\end{tikzpicture} 
\end{equation}

\begin{rem}\label{rem:dsets-cov}
It is proved in \cite{Cisinski-Moerdijk2} that there are left Quillen equivalences in both directions between complete dendroidal Segal spaces and dendroidal sets with the operadic model structure,
\[
\dS_{oper} \xrightarrow{\sim} \dSp_{RSC} \xrightarrow{\sim} \dS_{oper}
\] 
Thus, the lower row in diagram (\ref{eq:plong}) is also equivalent to a row of localisations
\[
\dS_{oper} \to \dS_{cov} \to \dS_{stable}
\]
The analogous Quillen equivalent rows for the categorical structures are
\[
\sSp_{RSC} \to \sSp_{R,cov} = \sSp_{R,stable} % \xrightarrow{\simeq} \Sp_{KQ}
\]
and
\[
\Sp_{cat} \to \Sp_{cov} = \Sp_{KQ} \;.
\]
\end{rem}

\subsection{The Yoneda lemma}
In order to state the main lemma, Lemma \ref{lem:yon} below, we need a few more definitions.

Given a tree $T$, let $C(T)$ be the groupoid whose objects are inclusions $T \hookrightarrow T^\sharp$ where the tree $T^{\sharp}$ is obtained from $T$ by attaching a corolla $C_n$ with $n \geq 0$ to each leaf of $T$. (In other words, the objects of $C(T)$ are in bijective correspondence with functions from $\lambda(T)$ to $\NN$. In particular, for any tree $T$ without leaves, $C(T)$ has exactly one object, namely the identity map $T \to T^{\sharp}$.) A morphism in $C(T)$ is an isomorphism under $T$. 

The assignment $T \mapsto C(T)$ is functorial, i.e. $C(-)$ is a dendroidal object in groupoids. To see this, let $\alpha$ be a morphism $S \to T$ in $\Omega$. Then each leaf $e$ of $S$ is mapped to an edge $\alpha(e)$ of $T$, thus by contracting all the inner edges in $T^\sharp$ above $\alpha(e)$ we obtain a tree $S^\sharp$ and a diagram
\begin{equation*}%\label{eq:squnder}
	\begin{tikzpicture}[descr/.style={fill=white}, baseline=(current bounding box.base)] ] 
	\matrix(m)[matrix of math nodes, row sep=2.5em, column sep=2.5em, 
	text height=1.5ex, text depth=0.25ex] 
	{
	T & T^\sharp \\
	S & S^\sharp \\
	}; 
	\path[right hook->,font=\scriptsize] 
		(m-1-1) edge node [auto] {$i$} (m-1-2);
	\path[right hook->,font=\scriptsize] 
		(m-2-1) edge node [auto] {$\alpha^*i$} (m-2-2);
	\path[<-,font=\scriptsize] 
		(m-1-1) edge node [left] {$\alpha$} (m-2-1);
	\path[<-,font=\scriptsize] 		
		(m-1-2) edge node [auto] {$i_*\alpha$} (m-2-2);
	\end{tikzpicture} 
\end{equation*}
By construction, $S^\sharp$ is obtained from $S$ by gluing corollas on leaves, and the map $i_*\alpha$ defines an injection $\lambda(S^\sharp) \to \lambda(T^\sharp)$. We leave to the reader the verification that $\alpha \mapsto \alpha^*i$ indeed defines a map of groupoids $C(T) \to C(S)$ that is functorial in $\alpha$.

\begin{defn}\label{defn:underoperad} 
Let $X$ be a Segal dendroidal space, and $\sigma : U \to X_\eta$ a map of simplicial sets. We define $\sigma/X$ as the dendroidal space whose value at a tree $T$ is
\[
(\sigma/X)_T := \hocolim{(T \hookrightarrow T^\sharp) \in C(T)} Z_{T^\sharp} \; ,
\]
and $Z_{T^\sharp}$ is defined as the homotopy pullback
\begin{equation*}%\label{eq:squnder}
	\begin{tikzpicture}[descr/.style={fil	=white}, baseline=(current bounding box.base)] ] 
	\matrix(m)[matrix of math nodes, row sep=2.5em, column sep=2.5em, 
	text height=1.5ex, text depth=0.25ex] 
	{
	Z_{T^\sharp} & X_{T^\sharp} \\
	\map(\lambda(T^\sharp), U) & \map(\lambda(T^\sharp), X_\eta) \\
	}; 
	\path[->,font=\scriptsize] 
		(m-1-1) edge node [auto] {$$} (m-1-2);
	\path[->,font=\scriptsize] 
		(m-2-1) edge node [auto] {$\sigma$} (m-2-2);
	\path[->,font=\scriptsize] 
		(m-1-1) edge node [auto] {$$} (m-2-1);
	\path[->,font=\scriptsize] 		
		(m-1-2) edge node [auto] {$$} (m-2-2);
	\end{tikzpicture} 
\end{equation*}
where the right-hand map is induced by the inclusion of the leaves in $T^\sharp$.
\end{defn}

For each object $T \hookrightarrow T^\sharp$ in $C(T)$, we obtain a projection map $X_{T^\sharp} \to X_{T}$, and these assemble into a map $\pi: \sigma/X \to X$ of dendroidal spaces. Note also that if $T$ is a tree without leaves, then the map $(\sigma/X)_T \to X_T$ is a weak equivalence.

\begin{rem}
When $X$ is a normal dendroidal set (i.e. a Reedy cofibrant dendroidal \emph{discrete} space), the value of $\sigma/X$ at a tree $T$ can be described as the set of equivalence classes of triples 
\[
(T \hookrightarrow T^\sharp, a , \xi)
\]
where $T \hookrightarrow T^\sharp$ is an object of $C_T$, $a \in X_{T^\sharp}$ and $\xi : \lambda(T^\sharp) \to U$ is a labelling of the leaves of $T^\sharp$ by elements of $U$, compatible with the dendrex $a$ in the sense that
\[
\sigma \circ \xi (\ell) = \ell^*(a) \in X_\eta
\]
where $\ell^* : X_T^\sharp \to X_\eta$ is the restriction along $\ell : \eta \to T^\sharp$. The equivalence relation on such triples $(i, a, \xi)$ is defined by identifying $$(i : T \hookrightarrow T^\sharp, a, \xi) \sim (j : T \hookrightarrow T^\prime, b, \phi)$$ if and only if there is an isomorphism $\tau : T^\sharp \to T^\prime$ under $T$ such that 
$\tau^* b = a$ and $\tau^* \phi = \xi$, where $\tau^* \phi$ is the composition $\xi \circ \lambda(\tau) : \lambda(T^\sharp) \to U$.
\end{rem}

\begin{lem} The map $\pi : \sigma/X \to X$ is a covariant fibration.
\end{lem}
\begin{proof}
The goal is to show that the square
\begin{equation}\label{eq:sqcov}
	\begin{tikzpicture}[descr/.style={fill=white}, baseline=(current bounding box.base)] ] 
	\matrix(m)[matrix of math nodes, row sep=2.5em, column sep=2.5em, 
	text height=1.5ex, text depth=0.25ex] 
	{
	(\sigma/X)_T & \map(\lambda(T), (\sigma/X)_{\eta}) \\
	X_T & \map(\lambda(T), X_{\eta}) \\
	}; 
	\path[->,font=\scriptsize] 
		(m-1-1) edge node [auto] {$$} (m-1-2);
	\path[->,font=\scriptsize] 
		(m-2-1) edge node [auto] {$$} (m-2-2);
	\path[->,font=\scriptsize] 
		(m-1-1) edge node [auto] {$$} (m-2-1);
	\path[->,font=\scriptsize] 		
		(m-1-2) edge node [auto] {$$} (m-2-2);
	\end{tikzpicture} 
\end{equation}
is homotopy cartesian. Let $x \in X_T$. Since homotopy colimits of simplicial sets are stable under homotopy base change, the homotopy fibre of the left-hand map is identified with 
\begin{equation}\label{eq:undercfib}
\hocolim{(T \hookrightarrow T^\sharp) \in C(T)} \textup{hofibre}_x ( Z_{T^\sharp} \to X_T)
\end{equation}
where the map on the display is the composite of the projection $Z_{T^\sharp} \to X_{T^{\sharp}}$ with the restriction $X_{T^{\sharp}} \to X_T$.

Now, any object $T \hookrightarrow T^\sharp$ in $C(T)$ fits into a square of inclusions
\begin{equation*}
	\begin{tikzpicture}[descr/.style={fill=white}, baseline=(current bounding box.base)] ] 
	\matrix(m)[matrix of math nodes, row sep=2.5em, column sep=2.5em, 
	text height=1.5ex, text depth=0.25ex] 
	{
	\coprod_{\ell \in \lambda(T)} \Omega[\eta] & \coprod_{\ell \in \lambda(T)} \Omega[C_{n_\ell}] \\
	\Omega[T] & \Omega[T^\sharp] \\
	}; 
	\path[->,font=\scriptsize] 
		(m-1-1) edge node [auto] {$$} (m-1-2);
	\path[->,font=\scriptsize] 
		(m-2-1) edge node [auto] {$$} (m-2-2);
	\path[->,font=\scriptsize] 
		(m-1-1) edge node [auto] {$$} (m-2-1);
	\path[->,font=\scriptsize] 		
		(m-1-2) edge node [auto] {$$} (m-2-2);
	\end{tikzpicture} 
\end{equation*}
where $n_\ell$ is the set of leaves of the corolla attached to the $\ell^{th}$ leaf of $T$. By the assumption that $X$ is a Segal dendroidal space, the induced square
\begin{equation*}%\label{eq:squnder}
	\begin{tikzpicture}[descr/.style={fill=white}, baseline=(current bounding box.base)] ] 
	\matrix(m)[matrix of math nodes, row sep=2.5em, column sep=2.5em, 
	text height=1.5ex, text depth=0.25ex] 
	{
	X_{T^{\sharp}} & X_T \\
	\prod_{\ell \in \lambda(T)} X_{C_{n_\ell}}  & \prod_{\ell \in \lambda(T)} X_{\eta} \\
	}; 
	\path[->,font=\scriptsize] 
		(m-1-1) edge node [auto] {$$} (m-1-2)
		(m-2-1) edge node [auto] {$$} (m-2-2)
		(m-1-1) edge node [auto] {$$} (m-2-1)
		(m-1-2) edge node [auto] {$$} (m-2-2);
	\end{tikzpicture} 
\end{equation*}
is homotopy cartesian. We can use this simple observation to  simplify the expression (\ref{eq:undercfib}), at the cost of losing functoriality in $T$. Recall the groupoid $\Sigma$ of finite sets and bijections. Clearly, the groupoid $C(T)$ breaks up as $\lambda(T)$-many copies of $\Sigma$. Given $\ell \in \lambda(T)$, let $x_\ell \in X_{\eta}$ be the image of $x$ by the map $X_{T} \to X_\eta$ which selects the $\ell^{th}$ leaf. Then the space (\ref{eq:undercfib}) is weakly equivalent to 
\[
%\prod_{\ell \in \lambda(T)} \hocolim{(n \in \Sigma} \; \textup{hofibre}_{x_\ell}(Z_{C_n} \xrightarrow{\text{root}} X_{\eta})
\hocolim{(n_1, \dots, n_\ell) \in \Sigma \times \dots \times \Sigma} \; \textup{hofibre}_{(x_1, \dots, x_\ell)}(Z_{C_{n_1}} \times \dots \times Z_{C_{n_\ell}} \xrightarrow{\text{root}} X_{\eta} \times \dots \times X_{\eta})
\]
and this is identified with the homotopy fibre of the right-hand map in square (\ref{eq:sqcov}). Thus, the square (\ref{eq:sqcov}) is homotopy cartesian.
\end{proof}

As before, let $\sigma : U \to X_\eta$ be a map of simplicial sets. We say that $\sigma$ (or, by abuse, $U$) is \emph{initial} if the map $\pi : \sigma/X \to X$ is a (levelwise) weak equivalence. (To assert this, since $\pi$ is a covariant fibration, it suffices to show that $\pi$ is a weak equivalence on colours.)

For example, any map $\sigma : U \to X_\eta$ has a canonical lift $s \sigma : U \to (\sigma/X)_\eta$ along $\pi : (\sigma/X)_\eta \to X_\eta$ which assigns to $u \in U$ the element of $(\sigma/X)_\eta$ given by the unary corolla $C_1$ and the degenerate element of $X_{C_1}$ given by $\sigma(u)$. Then the map $s\sigma/(\sigma/X) \to \sigma/X$ is an isomorphism at $\eta$. Therefore $s\sigma$ is initial.

\begin{lem}\label{lem:yon}
Suppose $\pi : X \to B$ is map between Segal dendroidal spaces and $\sigma : U \to X_{\eta}$ a map which is initial in $X$. Then the map
\[
\Omega[\eta] \otimes U \to X
\]
adjoint to $\sigma$, is a weak equivalence in the covariant model structure over $B$.
\end{lem}

Before we go into the proof, we isolate an easy result which will be used there.
\begin{lem}\label{lem:yonES}
Suppose $v : Y \to B$ is a covariant fibration and $\beta : U \to Y_\eta$ a map of spaces. Then the induced map $\beta/Y \to v \beta/B$ is a weak equivalence.
\end{lem}
\begin{proof}
Since $\pi$ is a covariant fibration, the map of hocolim-diagrams defining $\beta/Y$ and $\pi\beta/B$ is a levelwise weak equivalence.
\end{proof}

\begin{proof}[Proof of Lemma \ref{lem:yon}] We need to show that for every covariant fibration $v : Y \to B$ the induced map
\begin{equation}\label{eq:yon}
\RR \map_B(X, Y) \to \RR \map_B(\Omega[\eta] \otimes U, Y)  
\end{equation}
is a weak equivalence, where $\RR \map_{B}(-,-)$ denotes the derived mapping space in the projective model category of dendroidal spaces over $B$. We may assume that $v$ is a levelwise fibration between levelwise fibrant dendroidal spaces.

The derived mapping space $\RR \map_B(X, Y)$ is identified with the homotopy pullback of the diagram
\[
* \xrightarrow{\pi} \RR \map(X,B) \xleftarrow{v \circ} \RR \map(X,Y)
\]
and, similarly, $\RR \map_B(\Omega[\eta] \otimes U, Y)$ is identified with the (homotopy) pullback of
\[
* \xrightarrow{\pi \sigma} \map(U, B_{\eta}) \xleftarrow{v \circ} \map(U, Y_{\eta}) \; .
\]
Now pick a basepoint $\beta \in \map(U,Y_{\eta})$ such that $v\beta = \pi \sigma$ in $\map(U, B_\eta)$. We use based-maps notation and write $\RR \map((X,\sigma), (Y,\beta))$ for the homotopy fibre of the composition
\[
\RR \map(X,Y) \to \RR \map(X_{\eta}, Y_{\eta}) \xrightarrow{\sigma^*} \RR \map(U,Y_{\eta})
\]
over $\beta$, and similarly for $\RR \map((X,\sigma), (B, \pi\sigma))$. By commuting homotopy limits, the homotopy fibre of (\ref{eq:yon}) over $\beta$ agrees with the homotopy fibre of 
\begin{equation}\label{eq:yonmainmap}
\RR \map((X,\sigma), (Y,\beta)) \xrightarrow{v \circ} \RR \map((X,\sigma), (B, \pi\sigma))
\end{equation}
over $\pi$. We recast the map (\ref{eq:yonmainmap}) as the left-hand vertical arrow in the commutative square below: 
\begin{equation}\label{eq:usq}
	\begin{tikzpicture}[descr/.style={fill=white}, baseline=(current bounding box.base)] ] 
	\matrix(m)[matrix of math nodes, row sep=2.5em, column sep=2.5em, 
	text height=1.5ex, text depth=0.25ex] 
	{
	\RR \map((X,\sigma), (Y,\beta)) & \RR \map((\sigma/X,s \sigma), (\beta/Y, s \beta)) \\
	\RR \map((X,\sigma), (B, v\beta)) & \RR \map((\sigma/X, s \sigma), (v\beta/B, s {v\beta})) \\
	}; 
	\path[->,font=\scriptsize] 
		(m-1-1) edge node [auto] {$(\star)$} (m-1-2)
		(m-2-1) edge node [auto] {$$} (m-2-2)
		(m-1-1) edge node [auto] {$v \circ $} (m-2-1)
		(m-1-2) edge node [auto] {$v \circ$} (m-2-2);
	\end{tikzpicture} 
\end{equation}
Recall that $s\sigma$ refers to the canonical lift $U \to (\sigma/X)_{\eta}$ of $\sigma$ along $(\sigma/X)_{\eta} \to X_\eta$ (and similary for $s\beta$ and $s v\beta$).

By Lemma \ref{lem:yonES}, we know that the right-hand map is a weak equivalence. In the remainder of the proof, we will show that the horizontal maps are weak equivalences.

Informally, the statement about the horizontal maps is a dendroidal version of the $1$-categorical triviality that a functor from a category $C$ to a category $D$ sending a fixed object $c \in C$ to a fixed object $d \in D$ is the same as a functor between undercategories $c/C \to d/D$ sending $id_{c}$ to $id_d$, provided $c$ is an \emph{initial} object of $C$. We now give more details, concentrating on the top horizontal map $(\star)$ since the argument for the lower one is identical. The composite map
\[
\RR \map((X,\sigma), (Y,\beta)) \xrightarrow{(\star)} \RR \map((\sigma/X,s \sigma), (\beta/Y, s \beta)) 
\xrightarrow{\pi_Y} \RR \map((\sigma/X,s \sigma), (Y, \beta))
\]
agrees with precomposition with $\pi_X : \sigma/X \to X$, and therefore it is a weak equivalence since $\sigma$ is initial. This means that $(\star)$ has a homotopy left inverse.

To simplify notation below, let us temporarily write $s\sigma/X$ as short-hand for $s\sigma/(\sigma/X)$, the under-dendroidal space of $\sigma/X$ under $s\sigma : U \to (\sigma/X)_{\eta}$. The composite map
\[
	\begin{tikzpicture}[descr/.style={fill=white}] 
	\matrix(m)[matrix of math nodes, row sep=2.5em, column sep=2.5em, 
	text height=1.5ex, text depth=0.25ex] 
	{
	\RR \map((\sigma/X,s\sigma), (\beta/Y,s\beta)) & \RR \map((\sigma/X,s\sigma), (Y,\beta))  \\
	 & \RR \map((s\sigma/X,ss\sigma), (\beta/Y,s\beta)) \\
	}; 
	\path[->,font=\scriptsize] 
		(m-1-1) edge node [auto] {} (m-1-2);
	\path[->,font=\scriptsize] 
		(m-1-2) edge node [auto] {} (m-2-2);
	\end{tikzpicture} 
\]
agrees with precomposition with $\pi : s\sigma/X \to \sigma/X$ and so is a weak equivalence (since $s\sigma$ is initial in $\sigma/X$). Because $\sigma$ is initial in $X$, the right-hand map is identified with $(\star)$ and we conclude that $(\star)$ also has a homotopy right inverse. 
\end{proof}

This lemma will be heavily used in the following sections, mostly in the form of a corollary that we now describe. 

\begin{cor}\label{cor:yon} Let $B$ be a Segal dendroidal space. Then for a map $\sigma : U \to B_{\eta}$ of spaces, the map $\Omega[\eta] \otimes U \to \sigma/B$ (corresponding to $s\sigma$) is a covariant weak equivalence over $B$.
\end{cor}
\begin{proof}
The map $s\sigma$ is initial so we may apply Lemma \ref{lem:yon} with $X = \sigma/B$.
\end{proof}

\section{Invariance under Dwyer-Kan equivalences}\label{sec:invariance}

Let $f : P \to Q$ be a map of Segal dendroidal spaces. Given a dendroidal space $X$ over $Q$, we obtain a dendroidal space $f^* X$ over $P$ by base change, i.e. taking the pullback of $X$ along $f$. This restriction map has a left adjoint $f_!$ which is given by composition, i.e.
\[
f_! ( Y \to P) = Y \to P \to Q
\]
for a dendroidal space $Y$ over $P$. In this section we prove

\begin{thm}\label{thm:invariance}
Let $f : P \to Q$ be a Dwyer-Kan equivalence between Segal dendroidal spaces. Then the adjoint pair
\[
f_! : \dSp{/P} \leftrightarrows \dSp{/Q} : f^*
\]
is a simplicial Quillen equivalence for the covariant (projective or Reedy) model structures.
\end{thm}

\begin{rem}
We emphasise that we do not assume $P$ and $Q$ to be \emph{complete} Segal dendroidal spaces. In fact, under this stronger assumption, $f$ is a weak equivalence between fibrant objects and by Brown's lemma we may assume that it is a trivial fibration, in which case the theorem is obvious.
\end{rem}

We will deduce Theorem \ref{thm:invariance} from the following two propositions.

\begin{prop}\label{prop:DK_ff}
Let $f : P \to Q$ be a fully faithful map between Segal dendroidal spaces and let $\sigma : U \to P_{\eta}$ be a map of spaces. Then the square
\begin{equation*}
	\begin{tikzpicture}[descr/.style={fill=white} ] 
	\matrix(m)[matrix of math nodes, row sep=2.5em, column sep=2.5em, 
	text height=1.5ex, text depth=0.25ex] 
	{
	\sigma/P & f\sigma/Q \\
	P &  Q \\
	}; 
	\path[->,font=\scriptsize] 
		(m-1-1) edge node [auto] {$$} (m-1-2)
		(m-2-1) edge node [auto] {} (m-2-2)
		(m-1-1) edge node [auto] {$$} (m-2-1)
		(m-1-2) edge node [auto] {$$} (m-2-2);
	\end{tikzpicture} 
\end{equation*}
is homotopy cartesian.
\end{prop}
\begin{proof}
We may assume that the map $f$ is a projective fibration between projectively fibrant objects. Indeed, the hypothesis, the conclusion and the construction involved are all invariant under levelwise weak equivalences. The vertical maps are covariant fibrations, and those are invariant under pullback, so the claim is that the induced map $\sigma/P \to f^* (f\sigma/Q)$ of covariant fibrations over $P$ is a weak equivalence. Since weak equivalences between covariant fibrations are determined on colours (Lemma \ref{lem:covwe}), it suffices to show that for every $x \in P_\eta$, the induced map on fibres over $x$ is a weak equivalence of spaces. This follows immediately from the assumption that $f$ is fully faithful.
\end{proof}

\begin{prop}\label{prop:pullback_conservative}
Let $f : P \to Q$ be an essentially surjective map between Segal dendroidal spaces. A map $u : X \to Y$ between covariant fibrations over $Q$ is a covariant weak equivalence if and only if the pullback $\RR f^* X \to \RR f^*Y$ over $P$ is a covariant weak equivalence.
\end{prop}
\begin{proof}
We may assume that $X \to Q$ and $Y \to Q$ are projective fibrations. We again make use of the fact (Lemma \ref{lem:covwe}) that a map between covariant fibrations is a covariant weak equivalence if and only if it is a weak equivalence on colours. Clearly, if $u$ is a weak equivalence on colours then the same holds for $f^*u$. Conversely, consider the diagram
\begin{equation*}
	\begin{tikzpicture}[descr/.style={fill=white} ] 
	\matrix(m)[matrix of math nodes, row sep=2.5em, column sep=2.5em, 
	text height=1.5ex, text depth=0.25ex] 
	{
	(f^*X)_\eta & X_\eta & Y_\eta \\
	P_\eta & Q_\eta & Q_\eta \\
	}; 
	\path[->,font=\scriptsize] 
		(m-1-1) edge node [auto] {$$} (m-1-2)
		(m-2-1) edge node [auto] {} (m-2-2)
		(m-1-1) edge node [auto] {$$} (m-2-1)
		(m-1-2) edge node [auto] {$$} (m-2-2)
		(m-1-2) edge node [auto] {$$} (m-1-3)
		(m-1-3) edge node [auto] {} (m-2-3);
	\path[->,font=\scriptsize] 
		(m-2-2) edge node [auto] {$id$} (m-2-3);
	\end{tikzpicture} 
\end{equation*}
in which the left-hand square is (homotopy) cartesian by definition, and the outer rectangle is (homotopy) cartesian by the assumption that $f^*u$ is a weak equivalence. Therefore, for any $p \in P_{\eta}$, the induced map between the vertical homotopy fibres of the right-hand square over $f(p)$ is a weak equivalence.

Now let $q \in Q_{\eta}$. Since $f$ is essentially surjective, there exists some $p$ such that $f(p)$ is homotopy equivalent to $q$. By Proposition \ref{prop:covfib_is_compl}, the homotopy fibres of $X_\eta \to Q_{\eta}$ over $q$ and $f(p)$ are weakly equivalent, and the same is true for $Y_\eta \to Q_\eta$. Therefore, the induced map between the vertical homotopy fibres of the right-hand square over $q$ is also a weak equivalence. In other words, the map $f$ is a weak equivalence on colours.
\end{proof}

\begin{proof}[Proof of Theorem \ref{thm:invariance}]
The adjunction is a Quillen pair since the left adjoint preserves cofibrations and $f^*$ preserves covariant fibrations.
To show that it is a Quillen equivalence, we proceed in two steps. The essential surjectivity of $f$ implies that $\RR f^*$ is homotopy conservative by Proposition \ref{prop:pullback_conservative}. In the remainder of the proof we show that the derived unit map is a weak equivalence. By the triangular identities for an adjunction and the homotopy conservativity of $\RR f^*$, it then follows that the derived counit is also a weak equivalence, thus proving the result.

The functor $\RR f^*$ commutes with homotopy colimits, since homotopy colimits in spaces are stable under homotopy base change. For the same reason, if $W_\bullet \to Q$ is a simplicial object in the category of dendroidal spaces (over $Q$) and each $W_n$ is a covariant fibration over $Q$, then $|W_\bullet| \to Q$ is a covariant fibration. Therefore, by Lemma \ref{lem:reduction}, it suffices to show that the derived unit map $Z \to \RR f^* \LL f_! Z$ is a weak equivalence for a cofibrant object $Z$ of the form
\[
\coprod_{i \in I} \Omega[T^i] \otimes \Delta^{n_i} \to P
\]
for some set $I$. Consider the diagram (over $P$)
\begin{equation*}
	\begin{tikzpicture}[descr/.style={fill=white}, baseline=(current bounding box.base)] ]
	\matrix(m)[matrix of math nodes, row sep=2.5em, column sep=2.5em, 
	text height=1.5ex, text depth=0.25ex] 
	{
	Z & \RR f^* f_! Z \\
	\Omega[\eta] \otimes C & \RR f^* f_!(\Omega[\eta] \otimes C)  \\
	}; 
	\path[->,font=\scriptsize] 
		(m-1-1) edge node [auto] {$$} (m-1-2)
		(m-2-1) edge node [auto] {$$} (m-2-2);
	\path[<-,font=\scriptsize] 
		(m-1-1) edge node [left] {$$} (m-2-1)
		(m-1-2) edge node [auto] {$$} (m-2-2);
	\end{tikzpicture} 
\end{equation*}
where $C := \coprod_{i \in I} \lambda(T^i) \simeq \coprod_{i \in I} \lambda(T^i) \times \Delta^{n_i}$. The vertical maps are by definition weak equivalences in the covariant model structure over $P$. Now let $\sigma : C \to P_{\eta}$ denote the map adjoint to $\Omega[\eta] \otimes C \to P$. By the Yoneda lemma (in the formulation of Corollary \ref{cor:yon}), the lower horizontal arrow is weakly equivalent to
\[
\sigma/P \to f^* (f\sigma/Q)
\]
(over $P$). We showed in Proposition \ref{prop:DK_ff} that this map is a levelwise weak equivalence. Therefore, the derived unit map is a weak equivalence.
\end{proof}

\section{Dendroidal spaces and algebras over an operad}\label{sec:algebras}
Let $P$ be a $\Sigma$-free coloured operad in simplicial sets. The category of algebras in simplicial sets carries a model structure transferred from the Kan-Quillen model structure on simplicial sets \cite{Berger-Moerdijk}. The aim of this section is to prove the following theorem, which relates this model category to the covariant model structure on dendroidal spaces.

\begin{thm}\label{thm:alg}
Let $P$ be a $\Sigma$-free operad in simplicial sets and let $NP$ be its dendroidal nerve. There is a natural Quillen equivalence
\[
\mathcal{F}_P : (\dSp{/NP})_{P, cov} \leftrightarrows \mathsf{Alg}(P) : \sG_P
\]
between the covariant localisation of the projective model structure on $\dSp{/NP}$ and the model category of $P$-algebras.
\end{thm}

We begin by describing the functors $\sG_P$ and $\mathcal{F}_P$ in the statement of the theorem. Because the operad $P$ is fixed throughout, we shall simply write $\sG$ and $\mathcal{F}$ for $\sG_P$ and $\mathcal{F}_P$.

\medskip

Recall that the nerve of a simplicial operad $P$ with set of colours $C$ is the dendroidal space defined by
\[
(NP)_T = \map(\Omega(T), P)
\]
where $\Omega(T)$ is the (free) coloured operad defined by the tree $T$. Elements of $(NP)_T$ can be thought of as labellings of the edges of $T$ by colours of $P$ and of the vertices by operations. In particular, any element $\xi \in (NP)_T$ restricts to a function $\lambda(\xi) : \lambda(T) \to C$ on the leaves of $T$.

\medskip
Given a $P$-algebra $A = \{(A(c)\}_{c \in C}$, the dendroidal space $\sG(A)$ is the nerve of a certain coloured operad $W$ (internal to simplicial sets, i.e. with a space rather than a set of colours). Colours of $W$ are pairs $(c, x)$ with $c \in C$ and $x \in A(c)$. For a collection of colours $c_1, \dots, c_n, d$, let $m$ denote the algebra multiplication map
\[
P(c_1, \dots, c_n;d) \times A(c_1) \times \dots A(c_n) \to A(d) \; .
\]
An operation in $W$ from $(c_1, x_1), \dots, (c_n, x_n)$ to $(d,y)$ exists if there exists an operation $z \in P(c_1, \dots, c_n;d)$ such that $y = m(z, x_1, \dots, x_n)$ in $A(d)$. From its definition, it is clear that $W$ comes equipped with a forgetful map to $P$.

\medskip
The nerve of $W$ is by definition $\sG(A)$. Thus,
\[
\sG(A)_{\eta} = \coprod_{c \in C} A(c)
\]
and, for an arbitrary tree $T$, the space $\sG(A)_T$ can be identified with the pullback
\[
(NP)_{T} \times_{\map(\lambda(T),(NP)_{\eta})} \map(\lambda(T),\sG(A)_{\eta})
\]
where $\lambda(T)$ is the set of leaves of $T$. This description makes it obvious that the map $\sG(A) \to NP$ is a covariant fibration and that the functor $\sG$ sends fibrations of $P$-algebras to levelwise fibrations. It is also clear that $\sG$ preserves and detects weak equivalences between arbitrary objects.

The functor $\sG$ admits a left adjoint $\mathcal{F}$, which is defined on objects of the form $\Omega[T] \otimes K \to NP$, for a space $K$ and a tree $T$, by 
\[
\mathcal{F}(\Omega[T] \otimes K \xrightarrow{\xi} NP) = \mathsf{Free}_P (\lambda(T) \times K \xrightarrow{\lambda(\xi)} C)\; ,
\]
the free $P$-algebra generated the simplicial set $\lambda(T) \times K$ over $C$, briefly denoted by $\mathsf{Free}_P(\lambda(\xi))$. This definition is clearly functorial with respect to the simplicial set $K$. To see that this assignment is also functorial with respect to the tree $T$, first note by using the operations of the operad $P$ given by the map $\xi$, we can assign an element of $\mathsf{Free}_P ({\lambda(\xi)})$ to each edge of $T$. Now, a morphism $f : S \to T$ over $NP$ sends each leaf $\ell$ of $S$ to an edge $f(\ell)$ in $T$, and hence induces a map 
$$\lambda(S) \times K \to \mathsf{Free}_P(\lambda(\xi))$$ 
respecting the colours. Hence, $f$ defines a $P$-algebra map from $\mathcal{F}(\Omega[S] \otimes K \to NP)$ to $\mathcal{F}(\Omega[T] \otimes K \to NP)$.

\begin{lem}\label{lem:counitGFrep}
Let $U$ be a space and $\sigma : U \to C$ a map. Then the unit morphism
\[
\Omega[\eta] \otimes U \longrightarrow \sG \mathcal{F}( \Omega[\eta] \otimes U \xrightarrow{\sigma} NP) 
\]
is a covariant weak equivalence.
\end{lem}

\begin{proof}
This is a consequence of Corollary \ref{cor:yon} because $\sG \mathcal{F}(\Omega[\eta] \otimes U \xrightarrow{\sigma} NP)$ coincides with $\sigma/NP$ if the operad $P$ is $\Sigma$-free.
\end{proof}

We shall often leave the maps to $NP$ implicit, so for example we write $\Omega[\eta] \otimes U$ for $\sigma : \Omega[\eta] \otimes U \xrightarrow{} NP$. An immediate consequence of the lemma is the following.

\begin{cor}\label{cor:T}
Given a tree $T$ and a map $\Omega[T] \otimes U \to NP$, the unit 
\[
\Omega[T] \otimes U \to \sG F(\Omega[T] \otimes U)
\] is a covariant weak equivalence over $NP$.
\end{cor}
\begin{proof}
This follows by considering the square
\begin{equation*}
	\begin{tikzpicture}[descr/.style={fill=white} ] 
	\matrix(m)[matrix of math nodes, row sep=2.5em, column sep=2.5em, 
	text height=1.5ex, text depth=0.25ex] 
	{
	\amalg_{\lambda(T)} \Omega[\eta] \otimes U & \sG \mathcal{F} (\amalg_{\lambda(T)} \Omega[\eta] \otimes U) \\
	 \Omega[T] \otimes U & \sG \mathcal{F} (\Omega[T] \otimes U) \\
	}; 
	\path[->,font=\scriptsize] 
		(m-1-1) edge node [auto] {$$} (m-1-2)
		(m-2-1) edge node [auto] {} (m-2-2)
		(m-1-1) edge node [auto] {$$} (m-2-1)
		(m-1-2) edge node [auto] {$$} (m-2-2);
	\end{tikzpicture} 
\end{equation*}
where the left-hand map is a covariant weak equivalence by definition, and the right-hand map is an isomorphism. Since $\amalg_{\lambda(T)} \Omega[\eta] \otimes U = \Omega[\eta] \otimes (U \times \lambda(T))$ the map on top is a weak equivalence by Lemma \ref{lem:counitGFrep} and hence so is the map on the bottom of the square.
\end{proof}

\begin{rem} A map $\Omega[T] \to \sG \mathcal{F} (\Omega[T])$ (over $NP$) need not be a Reedy cofibration (and it is never a projective cofibration). Indeed, the functor $\mathcal{F}(-)$ does not preserve monomorphisms as one easily sees by applying it to $\eta \to \overline{\eta}$ where $\overline{\eta}$ denotes the tree with a single edge and a single vertex.
\end{rem}

\begin{proof}[Proof of Theorem \ref{thm:alg}]
We have already observed that $(\mathcal{F},\sG)$ form a Quillen pair and that $\sG$ preserves and detects weak equivalences between arbitrary objects. In particular, $\sG$ is homotopy conservative and is already derived. So to accomplish our task we will show that, for every \emph{projectively cofibrant} dendroidal space $X$ over $NP$, the unit map
\begin{equation}\label{eq:alg-unit}
X \to \sG \mathcal{F} (X) 
\end{equation}
is a covariant weak equivalence over $NP$. Lemma \ref{lem:reduction} applies as the assumptions there are trivially satisfied. So it suffices to prove this for $X$ of the form $\Omega[T] \otimes K$, which we have already established in Corollary \ref{cor:T}.
\end{proof}

\begin{cor}[Heuts \cite{Heuts}]
Let $P$ be a $\Sigma$-free simplicial operad and let $w^*(P)$ be its homotopy coherent nerve. There is a zig-zag of Quillen equivalences between $\dS/w^*(P)$ equipped with the covariant model structure and the model category of simplicial $P$-algebras.
\end{cor}
\begin{proof}
For any complete Segal dendroidal space $X$, the Quillen equivalence between dendroidal sets and dendroidal spaces of Remark \ref{rem:dsets-cov} induces a Quillen equivalence between dendroidal sets over $X_0$ and dendroidal spaces over $X$ for the operadic model structure and the complete Segal structure, respectively. The same holds for the covariant model structures, since the covariant model structure on dendroidal spaces is a localisation of the complete one (Proposition \ref{prop:covfib_is_compl}).

By the invariance theorem, if $X$ is a Segal dendroidal space which is not necessarily complete, the covariant model structure on dendroidal spaces over $X$ is Quillen equivalent to the covariant model structure on dendroidal sets over $\widehat{X}_0$, where $\widehat{X}$ is a (Rezk) completion of $X$. The corollary now follows from Theorem \ref{thm:alg} and the observation that $w^*(P)$ is equivalent to $\widehat{NP}_0$ (see \cite{Cisinski-Moerdijk3}).
\end{proof}

\section{$\Gamma$-spaces and dendroidal spaces}\label{sec:gamma-vs-dend}

We let $\FF$ denote the category of finite sets and partial maps. (That is, pairs $(A, f)$ where $A$ is a finite set and $f$ is a function $A^{\prime} \to B$ with $A^{\prime} \subset A$.) This category is isomorphic to the category of \emph{finite pointed sets} and pointed maps, by omission of the base point. We will refer to a covariant functor from $\FF$ to simplicial sets as a $\Gamma$-space (the category $\Gamma$ is by definition the opposite of $\FF$.) The category of all such will be denoted $$\GammaSp \,.$$
This category was introduced by Segal \cite{Segal}. In addition, we rely on \cite{Bousfield-Friedlander}. However, these authors work with functors into \emph{pointed} simplicial sets. We will comment on the inessential difference in the next section.

\quad The objects of $\FF$ will typically be denoted by capital letters $A, B, C, \dots$ and a morphism from $A$ to $B$ will be written as $A \pto B$ to emphasize partiality. Sometimes we will tacitly just work with a skeleton of $\FF$, whose objects are $\underline{n} = \{1, \dots, n\}$, $n \geq 0$. The action of a morphism $\alpha : A \pto B$ on a $\Gamma$-space is denoted $\alpha_* : X(A) \to X(B)$.

\medskip
Let us write
\[
\lambda : \Omega \to \Gamma
\]
for the functor which assigns to a tree $T$ its \emph{set of leaves} $\lambda(T)$. Its effect on a morphism $\alpha : S \to T$ is the partial function $\lambda(\alpha) : \lambda(T) \pto \lambda(S)$ defined as follows: for leaves $d \in \lambda(S)$ and $e \in \lambda(T)$ one sets $\lambda(\alpha)(e) = d$ if the path in $T$ from $e$ down to the root passes through $\alpha(d)$; and $\lambda(\alpha)(e)$ is undefined otherwise (this happens if that path misses the image of $\alpha$ entirely). This is indeed a (partial) function since for a pair of leaves $a, b \in \lambda(S)$ a path in $T$ from a leaf down to the root intersects \emph{at most one} of the edges $\alpha(a)$ and $\alpha(b)$. For the effect of this functor on generating morphisms of $\Omega$, note that $\lambda$ sends isomorphisms, degeneracies and inner face maps all to isomorphisms, top faces to total maps, but root faces to possibly partial maps.

\medskip
The following well-known definitions are due to Segal.
\begin{defn}
A $\Gamma$-space $X$ is \emph{special} if 
\begin{enumerate}
\item[$(i)$] The map $X(\varnothing) \to \Delta^0$ is weak equivalence.
\item[$(ii)$] For any two finite sets $A$ and $B$, the map $X(A \amalg B) \to X(A) \times X(B)$ induced by the inert surjections $A \pgets A \amalg B \pto B$ is a weak equivalence.
\end{enumerate}
Moreover, $X$ is \emph{very special} if it is special and the map
\[
X(\underline{2}) \to X(\underline{1}) \times X(\underline{1})
\]
induced by the total map $\underline{2} \to \underline{1}$ and one of the two inert surjections $\underline{2} \pto \underline{1}$, is a weak equivalence.
\end{defn}

By the \emph{special (or very special projective) $\Gamma$-space model structure} we mean the left Bousfield localisation of the projective model structure on $\GammaSp$ wherein the fibrant objects are the projectively fibrant $\Gamma$-spaces which are special (or very special, respectively). There are also Quillen equivalent special and very special Reedy $\Gamma$-space model structures which we will exploit in the proof of our Barratt-Priddy-Quillen theorem.

Given a $\Gamma$-space $X$, we let $\lambda^*(X)$ be the dendroidal space obtained by precomposition with $\lambda$. Then $\lambda^*$ has a left adjoint $\lambda_!$ which is uniquely characterised by the equation $\lambda_!(\Omega[T] \otimes K) = \FF(\lambda(T),-) \otimes K$, for any tree $T$ and space $K$.

\begin{thm}\label{thm:gamma}
The adjunction
\[
\lambda_! : \dSp \leftrightarrows \GammaSp : \lambda^*
\]
is a Quillen equivalence between the covariant projective model structure on $\dSp$ (over a point) and the special projective $\Gamma$-space model structure.
\end{thm}

\begin{proof}
The functors $\lambda_!$ and $\lambda^*$ form an adjoint pair for the projective model structures. The pair is also Quillen for the localised model structures since $\lambda^*$ clearly still preserves fibrant objects there. Moreover, the functor $\lambda^*$ is homotopy conservative, and so it suffices to verify that the derived unit map 
\[
Z \to \RR \lambda^* \LL \lambda_! Z
\]
is a covariant weak equivalence (over a point) for every projectively cofibrant dendroidal space $Z$. We may apply Lemma \ref{lem:reduction} to a resolution described just above this lemma to deduce that it is enough to prove that the derived unit map is a weak equivalence for $Z$ of the form $\Omega[T] \otimes \Delta^n$ or a coproduct of such objects. The conditions of the lemma are satisfied because $\lambda^*$ agrees with its derived functor (at the level of projective model structures) and it commutes with all colimits; moreover, since realization commutes with products, the geometric realisation of a simplicial object which is degreewise a special $\Gamma$-space is a special $\Gamma$-space.

We now proceed in three steps: $(i)$ $Z = \Omega[T]$, $(ii)$ $Z = \Omega[T] \otimes \Delta^n$, and $(iii)$ arbitrary coproducts of such. In case $(i)$, the unit map takes the form $\Omega[T] \to \RR \lambda^* \FF(\lambda(T), -)$. Writing $L$ for $\lambda(T)$, the target of this map is identified with
\[
\lambda^*(B\Sigma^L)
\]
by the Barratt-Priddy-Quillen theorem \ref{thm:SegalL}. Moreover, by definition, the map from $\Omega[\eta] \otimes L$ to $\Omega[T]$ selecting the leaves is a covariant weak equivalence. So the claim is that the map $$\Omega[\eta] \otimes L \to \lambda^*(B\Sigma^L)$$ is a covariant weak equivalence. This is a consequence of the Yoneda Lemma (Lemma \ref{lem:yon}).

Case $(ii)$ follows automatically from case $(i)$ since $\Omega[T] \otimes \Delta^n$ is weakly equivalent to $\Omega[T]$.

As for $(iii)$, the case of \emph{finite} coproducts is straightforward. Indeed, let $Z$ be a $I$-indexed coproduct of objects of the form $\Omega[T_i]$, $i \in I$. Then the derived unit map is identified with
\[
\Omega[\eta] \otimes L \to \lambda^*(B\Sigma^L)
\]
where $L$ is the set $\coprod_{i \in I} \lambda(T_i)$. Again, this map is a weak equivalence by the Yoneda lemma. For an infinite set $I$, one can express it as a filtered colimit of its finite subsets and, using that $\lambda^*$ commutes with colimits and that a filtered colimit of special $\Gamma$-spaces is special, reduce the problem to the finite case.
\end{proof}

\begin{cor}[c.f. \cite{Segal}, \cite{May-Thomason} and \cite{Mandell}]
There is a Quillen equivalence between the special $\Gamma$-space model structure and the model category of $E_{\infty}$-spaces.
\end{cor}
\begin{proof}
Let $P$ be an $E_{\infty}$-operad, for example the Barratt-Eccles operad. Then $NP \to *$ is a trivial fibration. Hence, the (projective) covariant model structure on $\dSp/NP$ is Quillen equivalent to that of $\dSp$ itself. The result now follows by combining this observation with Theorems \ref{thm:gamma} and \ref{thm:alg}.
\end{proof}

Since $\lambda_!$ sends the localising maps for the stable model structure on $\dSp$ to the localising maps for the very special model structure, we obtain the following corollary.

\begin{cor}\label{cor:gamma-very}
The adjunction $(\lambda_!, \lambda^*)$ is a Quillen equivalence between the stable (projective) model structure on $\dSp$ and the very special (projective) $\Gamma$-space model structure.
\end{cor}

By composing the Quillen equivalences of Theorem \ref{thm:gamma} and Corollary \ref{cor:gamma-very} with the Quillen equivalences of Remark \ref{rem:dsets-cov} we also obtain the following corollary.

\begin{cor}
There is a zig-zag of Quillen equivalences between the covariant model structure on $\dS$ and the special model structure on $\GammaSp$, which localises to a further Quillen equivalence between the stable model structure on $\dS$ and the very special model structure on $\GammaSp$.
\end{cor}

\section{Pointed and reduced $\RR$-spaces}

The point of this section is to show that the we get equivalent homotopy theories of special $\Gamma$-spaces, independently of whether we require the local objects $X$ to satisfy $X(\varnothing)$ contractible (as we do) or $X(\varnothing)$ a point.

More generally, let $\RR$ be a generalised Reedy category with a terminal object $t$, and assume $d(t) = 0$. Consider the category $\RR$-spaces of simplicial presheaves on $\RR$. An $\RR$-space $X$ is called \emph{pointed} if it takes values in pointed simplicial sets. This is equivalent to $X$ being equipped with a map $x_0 : \RR(-,t) \to X$, i.e. a vertex $x_0 \in X(t)$. The $\RR$-space $X$ is called \emph{reduced} if $X(t)$ is a point, and \emph{weakly reduced} if $X(t)$ is weakly contractible. 

We denote by 
\[
(\RRSp)_{wred}
\]
the left Bousfield localisation of the Reedy model structure whose fibrant objects are weakly reduced, Reedy fibrant $\RR$-spaces. It is given by the localising cofibrations of the form
\[
\RR(-,t) \otimes \partial \Delta[n] \hookrightarrow \RR(-,t) \otimes \Delta[n]
\]
for any $n \geq 0$. We denote by 
\[
\RRSp_{*}
\]
the category of pointed $\RR$-spaces. It is the slice category $\RR(-,t) / \RRSp$ and it inherits a model structure from the Reedy model structure on $\RRSp$, for which the forgetful functor
\begin{equation}\label{eq:pointed-forgt}
\RRSp_{*} \to \RRSp
\end{equation}
preserves and detects fibrations, cofibrations and weak equivalences. The cofibrant objects are the objects $(X,x_0)$ for which $x_0 : \RR(-,t) \to X$ is a cofibration, i.e. the "well-pointed" cofibrant $\RR$-spaces.

\begin{prop}
The forgetful functor in (\ref{eq:pointed-forgt}) induces a right Quillen equivalence
\[
(\RRSp_*)_{wred} \to (\RRSp)_{wred}
\]
\end{prop}

\begin{proof}
Any morphism $f : X \to Y$ in a model category $\mathcal{E}$ induces a Quillen pair
\[
f_! : X/\mathcal{E} \leftrightarrows Y/\mathcal{E}  : f^*
\]
where $f^*$ is given by composition and $f_!$ by pushout. If $f: X \to Y$ is a weak equivalence between cofibrant objects (or just a weak equivalence, in case $\mathcal{E}$ is left proper), this Quillen pair is a Quillen equivalence as one easily checks. This proposition is just a special case, where $\mathcal{E}$ is $(\RRSp)_{wred}$ and $X = \varnothing$ is the initial object while $Y = \RR(-,t)$.
\end{proof}

Now consider the functors (defined below)
\[
\rho_! : \RRSp_{*} \leftrightarrows (\RRSp)_{red} : \rho^*
\]
between pointed $\RR$-spaces and reduced $\RR$-spaces. The functor $\rho^*$ is simply the forgetful functor. Its left adjoint $\rho_!$ is defined for a pointed $\RR$-space $(X,x_0)$ and any object $r$ in $\RR$ by the pushout of simplicial sets
\begin{equation}\label{eq:red-pushout}
	\begin{tikzpicture}[descr/.style={fill=white}, baseline=(current bounding box.base)] ] 
	\matrix(m)[matrix of math nodes, row sep=2.5em, column sep=2.5em, 
	text height=1.5ex, text depth=0.25ex] 
	{
	X(t) & X(r) \\
	\Delta[0] & \rho_!(X)(r)\\
	}; 
	\path[->,font=\scriptsize] 
		(m-1-1) edge node [auto] {$\pi_r^*$} (m-1-2)
		(m-2-1) edge node [auto] {$$} (m-2-2)
		(m-1-1) edge node [auto] {$$} (m-2-1)
		(m-1-2) edge node [auto] {$$} (m-2-2);
	\end{tikzpicture} 
\end{equation}
where $\pi_r$ is the unique map to the terminal object $t$.

Incidentally, if $t$ is a \emph{zero-object}, any $r$ also admits a unique section $\tau_r : t \to r$ of $\pi_r$, and the functor $\rho^*$ also has a right adjoint $\rho_*$, defined analogously as a pullback of $\tau_r^*$ and $x_0$.

\begin{prop}\label{prop:pointed-vs-red}
The Reedy model structure on pointed $\RR$-spaces induces a model structure on reduced $\RR$-spaces for which the functors $\rho^*$ detects fibrations and weak equivalences; in particular, it makes the adjoint pair $\rho_!$ and $\rho^*$ into a Quillen pair.
\end{prop}
\begin{proof}
This is an instance of the usual transfer along the adjoint pair $\rho_!$ and $\rho^*$. The transfer conditions are fulfilled because $\rho^*$ commutes with colimits, and $\rho^* \rho_!$ maps the generating trivial cofibrations in $\RRSp_{*}$ to weak equivalences. Indeed, these generating trivial cofibrations are of the form $X \amalg U \to X \amalg V$ where $X = \RR(-,t)$ and $U \to V$ is
\begin{equation}\label{eq:UtoV}
\RR(-,r) \otimes \Lambda^k[n] \cup \partial \RR(-,r) \otimes \Delta[n] \to \RR(-,r) \otimes \Delta[n]
\end{equation}
If $d(r) > 0$, then $U(t) \to V(t)$ is an isomorphism so $\rho^* \rho_! U \to \rho^* \rho_! V$ is a pushout of $U \to V$. And if $d(r) = 0$ then $r$ is isomorphic to $t$, in which case $\pi_s^* : U(t) \to U(s)$ is an isomorphism for any $s$ in $\RR$, whence $\rho^* \rho_!(X \amalg U) = X$. The same applies to $V$. In other words, for $r = t$ the functor $\rho^* \rho_!$ maps (\ref{eq:UtoV}) to an isomorphism.
\end{proof}

Let us say that $\RR$ has a \emph{projective} terminal object $t$ if for every object $r$ in $\RR$ the unique morphism $r \to t$ has a section. For example, this is the case if $t$ is a zero object as in the category $\Gamma$, or if every morphism in $\RR^-$ has a section as in Lemma \ref{lem:reedy-cof}. This condition implies that for every presheaf $X$ of (simplicial) sets on $\RR$, the map $\pi_r^* : X(t) \to X(r)$ is injective.

\begin{prop}
If $\RR$ has a projective terminal object, then the Quillen pair $\rho_!$ and $\rho^*$ of Proposition \ref{prop:pointed-vs-red} restricts to a Quillen equivalence
\[
(\RRSp_*)_{wred} \leftrightarrows (\RRSp)_{red}
\]
between weakly reduced pointed $\RR$-spaces and reduced $\RR$-spaces.
\end{prop}

\begin{proof}
Since $\rho^*$ maps fibrant objects to local objects, the Quillen pair of Proposition \ref{prop:pointed-vs-red} factors through the localisation as 
\begin{equation*}
	\begin{tikzpicture}[descr/.style={fill=white}, baseline=(current bounding box.base)] ] 
	\matrix(m)[matrix of math nodes, row sep=2.5em, column sep=2.5em, 
	text height=1.5ex, text depth=0.25ex] 
	{
	\RRSp_*  & (\RRSp_*)_{wred} \\
	(\RRSp)_{red} & \\
	}; 
	\path[->,font=\scriptsize] 
		(m-1-1) edge node [auto] {} (m-1-2);
	\path[dashed,<-,font=\scriptsize] 
		(m-2-1) edge node [auto] {$\overline{\rho}_!$} (m-1-2);
	\path[->,font=\scriptsize] 
		(m-1-1) edge node [left] {$\rho_!$} (m-2-1);
	\end{tikzpicture} 
\end{equation*}
where the arrows denote the left Quillen functors. To show that $\overline{\rho}_!$ and its right adjoint $\overline{\rho}^*$ form a Quillen equivalence it suffices to prove that the derived unit $X \to \RR \overline{\rho}^* \LL \overline{\rho}_! X$ is a a weak equivalence for every $X$. Because $\overline{\rho}^*$ detects weak equivalences, the fact that the derived counit is also a weak equivalence then follows. Also, the forgetful functor $\overline{\rho}^*$ preserves weak equivalences between arbitrary objects, so in fact it suffices to prove for a cofibrant and weakly reduced pointed $\RR$-space $(X,x_0)$ that the unit $X \to \rho^* \rho_! X$ is a weak equivalence. Consider the defining pushouts (\ref{eq:red-pushout}) of $\rho_! X$. Since $X(t) \to X(r)$ is a mono by the assumption that $t$ is small, and the map $X(t) \to \Delta[0]$ is a weak equivalence (by the assumption on $X$), we conclude that $X \to \rho^* \rho_! (X)$ is a weak equivalence by the left properness of simplicial sets.
\end{proof}

\begin{rem}
Using the slightly stronger assumption that $\RR$ has a zero-object $t$ (again of Reedy degree $d(t) = 0$), the functor $\rho^*$ has a right adjoint $\rho_*$. In fact, in this case $\rho^*$ is also left Quillen and induces a Quillen equivalence
\[
\overline{\rho}^* : (\RRSp)_{red} \to (\RRSp_*)_{wred} \; .
\]
Indeed, it is sufficient to prove that $\rho^* \rho_!$ sends the generating (trivial) cofibrations to (trivial) cofibrations and the localising maps $\RR(-,t) \otimes \partial \Delta[n] \to \RR(-,t) \otimes \Delta[n]$ to weak equivalences. This was essentially done already in the proof of Proposition \ref{prop:pointed-vs-red}.
\end{rem}

\appendix

\section{A special Barratt-Priddy-Quillen theorem}

In this appendix, we will give a proof of our version of the Barratt-Priddy-Quillen theorem, stated as Theorem \ref{thm:introBPQ} in the introduction.

The category $\Gamma$ is a generalised Reedy category, thus providing the category of $\Gamma$-spaces with a Reedy model structure, Quillen equivalent to the projective model structure. We denote the left Quillen functor defining this equivalence by
\[
\GammaSp_P \to \GammaSp_R
\]
with $P$ for projective and $R$ for Reedy, as before. However, we shall mostly work with $\Gamma$-spaces as \emph{covariant} functors
\[
X : \FF \to \Sp
\]
from the category $\FF$ of finite sets and partial maps. The Reedy structure on $\Gamma$ is defined by letting the degree of a finite set be its cardinality, while $\Gamma^-$ consists of duals of totally defined \emph{injections} $A \hookrightarrow B$, and $\Gamma^+$ consists of the duals of possibly partially defined surjections $A \pto B$. The usual surjection-injection factorization of a map between finite sets provides a factorization as a negative morphism followed by a positive one in the dual category $\Gamma$. With these definitions, a map $Y \to X$ between covariant functors $\FF \to \Sp$ is a Reedy fibration if and only if for any finite set $A$, the map
\[
Y(A) \to X(A) \times_{X(\partial A)} Y(\partial A)
\]
is a Kan fibration, where 
\[
X(\partial A) = \lim_{A \pto B \; epi} X(B)
\]
is the limit over all proper, possibly partially defined, surjections $A \pto B$. Furthermore, a map $X \to Y$ is a Reedy cofibration if each $X(A) \to Y(A)$ is injective and $Aut(A)$ acts freely on the complement of
\[
X(A) \cup \bigcup_{B \hookrightarrow A} Y(B) \to Y(A)
\] 
where the union is over all proper monomorphisms (totally defined) from $B$ to $A$. In particular, an object $Y$ is Reedy cofibrant if $Aut(A)$ acts freely on the complement of 
\[
\bigcup_{B \subset A} Y(B) \hookrightarrow Y(A)
\]
the union being over \emph{proper} subsets $B$ of $A$.

\medskip
Another category that will be relevant for us is the category $\MM$ of finite sets and (totally defined) monomorphisms. This category is also a generalised Reedy category. The degree of an object is again its cardinality, and every morphism is positive (i.e. $\MM^+ = \MM$) while $\MM^-$ consists of isomorphisms only. We write $\MMSp_{R}$ for the category of simplicial presheaves on $\MM$ with this Reedy model structure. For easy reference, we state explicitly that a map $Y \to X$ of $\MM$-spaces is a Reedy fibration if for any finite set $A$, the map
\[
Y(A) \to X(A) \times_{X(\partial A)} Y(\partial A)
\]
is a Kan fibration, where
\[
X(\partial A) = \lim_{B \subset A} X(B)
\]
is the limit over proper subsets $B$ of $A$. Furthermore, a map $X \to Y$ is a Reedy cofibration if each $X(A) \to Y(A)$ is a monomorphism of simplicial sets with $Aut(A)$ acting freely on $Y(A) - X(A)$.

\medskip
There is an inclusion functor 
\[
inv : \MM^{\op} \to \FF
\]
which sends an injection $m : B \hookrightarrow A$ to the partial map given by taking the inverse $m^{-1} : A \pto B$. Maps in the image of this functor are called \emph{inert}. (These are the partial maps for which the preimage of each singleton is a singleton.) Restriction along $inv$ defines a functor
\[
inv^* : \GammaSp \to \MMSp \;.
\]
There is yet another functor
\[
pow : \GammaSp \to \MMSp
\]
which assigns to a $\Gamma$-space $X$ the $\MM$-space $A \mapsto X(\underline{1})^A$ (here we use $X(\underline{1})^A$ as alternative notation for the iterated pullback over $X(\varnothing)$). 

The map $\rho_A : X(A) \to X(\underline{1})^A$ is actually a natural transformation $\rho : inv^*(X) \to pow(X)$.

\begin{lem}
If $X$ is a Reedy fibrant object in $\GammaSp$ then $\rho : inv^*(X) \to pow(X)$ is a Reedy fibration in $\MMSp$. It is a trivial fibration if $X$ is special.
\end{lem}
\begin{proof}
Take an object $A$ in $\MM$. Write $Z$ for $inv^*(X)$. We have to check that 
\[
Z(A) \to Z(\partial A) \times_{Z(\underline{1})^{\partial A}} Z(\underline{1})^{A}
\]
is a fibration, where $Z(\partial A)$ denotes the limit, over all proper subsets $B$ of $A$, of $Z(B)$, and similarly for $Z(\underline{1})^{\partial A}$. But $Z(\underline{1})^{\partial A} \cong Z(\underline{1})^A$, and so we have to check that
\begin{equation}\label{eq:mfib}
Z(A) \to \lim_{B \subset A} Z(B)
\end{equation}
is a fibration. If $A = \varnothing$, this is the fibration $X(\varnothing) \to *$. In general, let $V$ be the set of all partial maps $f : A \pto C$ which are proper surjections but are not total. Then the inert maps $A \pto B$ for $B\subset A$ are cofinal in $V$ and, in the notation of section \ref{sec:reedyref}, the map (\ref{eq:mfib}) is $X(A) \to X(\partial_V A)$. Lemma \ref{lem:Vreedy} completes the proof.
\end{proof}

We now turn to the statement and proof of our Barratt-Priddy-Quillen theorem. The groupoid of finite sets and \emph{bijections} will be denoted $\Sigma$. Moreover, for a finite set $A$, we shall write $ \Sigma(A) $ for the groupoid of finite sets over $A$ and fibrewise bijections between them. Thus, objects of $\Sigma(A)$ are (total) functions $f : U \to A$ between finite sets, and morphisms from $(f : U \to A)$ to $(g:V\to A)$ are bijections $\sigma : U \to V$ with $g \sigma = f$. Clearly, $\Sigma(-)$ is a covariant functors from $\FF$ to groupoids. Indeed, for a partial map $g : A \pto B$, the induced functor $g_* : \Sigma(A) \to \Sigma(B)$ sends a function $f : U \to A$ to the function obtained by restricting the composition $g \circ f$ to where it is defined.

Let $B\Sigma$ be the functor $\FF \to \Sp$ which sends a finite set $A$ to the nerve of $\Sigma(A)$. Thus, an $n$-simplex in $B\Sigma(A)$ is a diagram of finite sets of the form
\[
U_0 \xrightarrow{\sigma_1} U_1 \xrightarrow{\sigma_2} \cdots \xrightarrow{\sigma_n} U_n \xrightarrow{f} A
\]
where the ${\sigma_i} ' s$ are isomorphisms over $A$. For brevity we denote such an $n$-simplex by $(\sigma; f)$ or $(\sigma_1, \dots, \sigma_n; f)$. Notation: we sometimes write $f_i$ for the composite $f \sigma_{n} \dots \sigma_{i+2} \sigma_{i+1}$ for $i < n$ and $f_n = f$.

\begin{thm}\label{thm:Segal}
The map $u : \FF(\underline 1, -) \rightarrow B \Sigma$ corresponding to the object $id : \underline{1} \to \underline{1}$ in $\Sigma(\underline 1)$ is a trivial Reedy cofibration between Reedy cofibrant objects in the localised model category $(\GammaSp)_{RS}$ ~.
\end{thm}

Because the identity functor is a Quillen equivalence between the Reedy and projective model structures, we obtain:

\begin{cor}
The map $\FF(\underline 1, -) \rightarrow B \Sigma$ is a fibrant replacement of $\FF(\underline 1, -)$ in the localised projective model structure $(\GammaSp)_{PS}$ ~.
\end{cor}

\begin{lem}\label{lem-BScof}
 The $\Gamma$-space $B\Sigma$ is Reedy cofibrant.
\end{lem}
\begin{proof}
We will write $Y$ for $B\Sigma$. Given $A \in \FF$, we need to check that $Aut(A)$ acts freely on the complement of $\bigcup_{B \subset A} Y(B)$ (union over \emph{proper} subsets $B$ of $A$) in $Y(A)$. That complement is identified with the nerve of the category of \emph{surjections} $U \twoheadrightarrow A$ and isomorphisms between them and clearly $Aut(A)$ acts freely on these (it acts freely on the objects).
\end{proof}

\begin{rem} The $\Gamma$-space $B\Sigma$ is fibrant in the localised projective model structure (it is a special $\Gamma$-space). If we replace $B\Sigma$ by a Reedy fibrant object, it also becomes fibrant in the corresponding Reedy structure.
\end{rem}

Recall that $\MM$ is the category of finite sets and \emph{monomorphisms}. Following Segal, we now describe how to extend $B\Sigma$ to a functor 
\[
B\Sigma_{(-)} : \MMSp \to \GammaSp  \; ,
\]
where $\MMSp$ denotes the category of contravariant functors from $\MM$ to simplicial sets. 

First, let $F$ be a presheaf of \emph{sets} on $\MM$. Define 
\[\Sigma_F : \FF \to \Cat\] 
($\Cat$ denotes the category of categories) as follows. The objects of $\Sigma_F(A)$ are pairs $(f : U \xrightarrow{} A, x)$ where $f$ is an object of $\Sigma(A)$ and $x \in F(U)$. A morphism from $(f: U \xrightarrow{} A, x)$ to $(g : V \xrightarrow{} A, y)$ is a morphism $\sigma$ in $\Sigma(A)$ with the property that $\sigma^*y = x$.

Note that the projection $\pi_A : \Sigma_F(A) \to \Sigma(A)$ is a fibred category with discrete fibres. If $\varphi : A \pto B$ is a morphism in $\FF$, we define a functor
\[
\varphi_* : \Sigma_{F}(A) \to \Sigma_{F}(B)
\]
by mapping the object $(f : U \to A,x)$ to $(\varphi f : f^{-1}(\textup{dom}\, \varphi) \to B, y)$, where $y$ is the restriction of $x$ to $f^{-1}(\textup{dom}\, \varphi)$. This is indeed a functor, which makes the projection $\pi_A$ natural in A, 
\begin{equation*}
	\begin{tikzpicture}[descr/.style={fill=white}, baseline=(current bounding box.base)] ] 
	\matrix(m)[matrix of math nodes, row sep=2.5em, column sep=2.5em, 
	text height=1.5ex, text depth=0.25ex] 
	{
	\Sigma_F(A) & \Sigma_F(B) \\
	\Sigma(A) & \Sigma(B) \\
	}; 
	\path[->,font=\scriptsize] 
		(m-1-1) edge node [auto] {$\varphi_*$} (m-1-2);
	\path[->,font=\scriptsize] 
		(m-2-1) edge node [auto] {$$} (m-2-2);
	\path[->,font=\scriptsize] 
		(m-1-1) edge node [auto] {$$} (m-2-1);
	\path[->,font=\scriptsize] 		
		(m-1-2) edge node [auto] {$$} (m-2-2);
	\end{tikzpicture} 
\end{equation*}
Moreover, the construction is obviously functorial in $F$.

More generally, if $F$ is a presheaf of \emph{simplicial sets} on $\MM$, the same construction applied degreewise gives a functor $\Sigma_F$ on $\FF$ with values in the category of categories internal to simplicial sets (simplicial categories). Composing this functor with the classifying space construction then yields a functor
\[
B \Sigma_F : \FF \to \Sp \, ,
\]
i.e. a $\Gamma$-space. For a finite set $A$, the $n$-simplices of $B \Sigma_F (A) = \textup{diag}~N(\Sigma_F(A))$ are given by triples $(\sigma, f, x)$ where $(\sigma; f) = (\sigma_1, \dots, \sigma_n; f)$ is an $n$-simplex in $B\Sigma(A)$, i.e. of the form
\[
U_0 \xrightarrow{\cong} U_1 \xrightarrow{\cong} \cdots \xrightarrow{\cong} U_n \xrightarrow{f} A \; ,
\]
and $x \in F(U_0)_n$. For a partial map $\varphi : A \pto B$, the map of simplicial sets $\varphi_* : B\Sigma_F(A) \to B\Sigma_F(B)$ is the map which sends an $n$-simplex $((\sigma; f), x) = (\sigma, f, x)$ to 
$(B\Sigma(\varphi)(\sigma;f),y) = (\tau, \varphi f, y)$ which consists of a string
\[
V_0 \xrightarrow{\tau_1} V_1 \xrightarrow{\tau_2} \cdots \xrightarrow{\tau_n} V_n \xrightarrow{\varphi f} A \; ,
\]
where $V_i = f_{i}^{-1}(\textup{dom}\, \varphi)$, $\tau_i$ is the restriction of $\sigma_i$ to $V_i$, and $y$ is the image of $x$ in $F(V_0)_n$ given by restriction along the inclusion $V_0 \subset U_0$.

\begin{lem}
For any simplicial presheaf $F$ on $\MM$, the $\Gamma$-space $B\Sigma_F$ is Reedy cofibrant.
\end{lem}
\begin{proof}
As in Lemma \ref{lem-BScof}.
\end{proof}

\begin{lem}
The construction of $B\Sigma_F$ is functorial with respect to the simplicial presheaf $F$. Moreover, if $F \to G$ is a weak equivalence (i.e. each $F(A) \to G(A)$ is), then the map $B\Sigma_F \to B\Sigma_G$ of $\Gamma$-spaces (over $B\Sigma$) is too.
\end{lem}
\begin{proof}
The simplicial category $\Sigma_F(A)$ is fibred over $\Sigma(A)$ with fibre $F(A)$ and $F(A) \to G(A)$ induces a weak equivalence of simplicial categories $\Sigma_F(A) \to \Sigma_G(A)$ over $\Sigma(A)$. From this, the statement is clear.
\end{proof}

\begin{rem} Suppose $F$ has the property that for each pair of finite sets $A$ and $B$ the map
$
F(A \amalg B) \to F(A) \times_{F(\varnothing)} F(B)
$
induced by the inclusions into the coproduct is a weak equivalence, and that $F(\varnothing)$ is contractible. Then $B\Sigma_F$ is a special $\Gamma$-space. (That is, up to Reedy or projective fibrant replacement, $B\Sigma_F$ is a fibrant object in the appropriate model structure described in section \ref{sec:review}.)
\end{rem}

\begin{proof}[Proof of Theorem \ref{thm:Segal}]
We will show that for any Reedy fibrant special $\Gamma$-space $X$ the map
\begin{equation}\label{eq:BPQmap}
\hom(B\Sigma, X) \xrightarrow{u^*} \hom(\FF(\underline{1}, -), X) \; ,
\end{equation}
is an isomorphism, where $\hom$ is taken in the homotopy category of $(\GammaSp)_{RS}$. This implies that $u$ is a weak equivalence since $B\Sigma$ and $\FF(\underline{1}, -)$ are Reedy cofibrant. Moreover, the map $u$ is a monomorphism into a cofibrant object so it is also a Reedy cofibration. Note  that $\hom(\FF(\underline{1}, -), X)$ is $\pi_0 X(\underline{1})$.

Any map $f : B\Sigma \to X$ gives a vertex $u^*(f) = f_{\underline{1}}(id_{\underline{1}})$ in $X(\underline{1})$. Vice versa, using an idea from Segal \cite{Segal}, we will construct a map $B\Sigma \to X$ in the homotopy category from any such vertex $v$.

\medskip
Let $\epsilon : \mathbf{X} \to inv^*(X)$ be a Reedy cofibrant replacement of $inv^*(X)$, so that we have a sequence of trivial Reedy fibrations between $\MM$-spaces
\[
\mathbf{X} \xrightarrow{\epsilon} inv^*(X) \xrightarrow{\rho} pow(X) \; .
\]
A vertex $v$ of $X(\underline{1})_0$ defines a map $v : * \to pow(X)$, and we write $\mathbf{X}_v$ for the fibre of $\rho \epsilon$ over $v$. Then $\mathbf{X}_v \to *$ is a trivial fibration and $\mathbf{X}_v$ is cofibrant (it is a subobject of the cofibrant $\mathbf{X}$). So there is a contracting homotopy
\[
H : \mathbf{X}_v \times \mathbf{X}_v \to (\mathbf{X}_v)^{\Delta[1]} \; ,
\]
a section of the endpoint fibration $(\mathbf{X}_v)^{\Delta[1]} \to \mathbf{X}_v \times \mathbf{X}_v$. Moreover, there is a canonical map of $\Gamma$-spaces
\[
\xi : B \Sigma_{\mathbf{X}_v} \to X \; ,
\]
defined as follows. For an object $A$ of $\FF$ and an $n$-simplex $(\sigma, f, x)$ of $(B \Sigma_{\mathbf{X}_v})(A)$, 
\[
\xi_A (\sigma, f, x) = (f_0)_* \epsilon(x) \in X(A)_n \; ,
\]
where $f_0 = f \sigma_n \dots \sigma_2 \sigma_1 : U_0 \to A$. One can check that $\xi_A$ (which depends on $v$) is a well-defined simplicial map that is natural in $A$.

\medskip
We are now ready to prove that the map (\ref{eq:BPQmap}) is indeed a bijection. First, given a vertex $v \in X(\underline{1})_0$, choose a point $\widetilde{v}$ in the fibre $\mathbf{X}_v(\underline{1})_0$ over $v$.
Then we have a diagram
\[
B \Sigma \xleftarrow{\pi} B \Sigma_{\mathbf{X}_v} \xrightarrow{\xi} X
\]
with $\pi(\widetilde{v}) = id_{\underline{1}}$ and $\xi(\widetilde{v}) = v$, in which $\pi$ is a weak equivalence because $\mathbf{X}_v$ is contractible. This diagram represents a map $\xi \pi^{-1} : B \Sigma \to X$ in the homotopy category with $u^*(\xi \pi^{-1}) = v$.

In the other direction, suppose we are given a map 
\[
B \Sigma \xrightarrow{\psi} X \;.
\]
(Any map in the homotopy category can be so represented because $B\Sigma$ is cofibrant and $X$ is assumed fibrant.) Let $v = \psi_{\underline{1}}(id_{\underline{1}})$ in $X(\underline{1})_0$~, i.e. $v = u^* (\psi)$. We will construct a homotopy $K$
\begin{equation*}
	\begin{tikzpicture}[descr/.style={fill=white}, baseline=(current bounding box.base)] ] 
	\matrix(m)[matrix of math nodes, row sep=2.0em, column sep=2.0em, 
	text height=1.5ex, text depth=0.25ex] 
	{
	B \Sigma	& & X \\
	 & B \Sigma_{\mathbf{X}_v} \\
	}; 
	\path[->,font=\scriptsize] 
		(m-1-1) edge node [auto] {$\psi$} (m-1-3)
		(m-2-2) edge node [auto] {$\pi$} (m-1-1)
		(m-2-2) edge node [below] {$\xi$} (m-1-3);
	 \node [rotate=90] at (0.05,0.1) {$\Rightarrow$};	
	
	\end{tikzpicture} 
\end{equation*}
between $\psi \pi$ and $\xi$, using the contracting homotopy $H$. To this end, first choose a lift $\Psi$ as in
\begin{equation*}
	\begin{tikzpicture}[descr/.style={fill=white}, baseline=(current bounding box.base)] ] 
	\matrix(m)[matrix of math nodes, row sep=2.0em, column sep=2.0em, 
	text height=1.5ex, text depth=0.25ex] 
	{
		& & \mathbf{X} \\
	W & inv^*(B \Sigma) & inv^*(X) \\
	}; 
	
	\path[->,font=\scriptsize] 
		(m-2-2) edge node [auto] {$\psi$} (m-2-3)
		(m-2-1) edge node [auto] {} (m-2-2);
	\path[->, dashed,font=\scriptsize] 
		(m-2-1) edge node [auto] {$\Psi$} (m-1-3);
	\path[->>,font=\scriptsize] 		
		(m-1-3) edge node [auto] {$\epsilon$} (m-2-3);
	\end{tikzpicture} 
\end{equation*}
where $W$ is the sub $\MM$-space of $inv^*(B\Sigma)$ consisting of those $n$-simplices $(\sigma,f)$ for which $f$ is an isomorphism. Such a lift exists because $W$ is Reedy cofibrant.

Notice that the map $\psi_A : B\Sigma(A) \to X(A)$ maps an $n$-simplex $(\sigma, f)$ of $B\Sigma(A)$ to the fibre over $v$ whenever $f$ (and hence all the $f_i$) are isomorphisms, and hence in this case $\Psi_A$ maps $(\sigma,f)$ into $\mathbf{X}_v$.

Now, for a general $n$-simplex $(\sigma,f,x)$ in $B\Sigma_{\mathbf{X}_v}(A)$, i.e. $(\sigma,f) \in B\Sigma(A)$ and $x \in \mathbf{X}_v(U_0)_n$, define
\[
K_A(\sigma, f, x) = f_* \epsilon H(\Psi(\sigma, id_{U_n}), (\sigma_n \dots \sigma_1)_*(x)) \; .
\]
We leave to the reader the verification that  $K_A(\sigma,f,x)$ is a path from $\psi_A(\sigma,f)$ to $\xi_A(\sigma,f,x)$ and that $K_A$ is a well defined map. 
\end{proof}

In the proof of Theorem \ref{thm:gamma}, we use a statement more general than Theorem \ref{thm:Segal}. This more general version, stated as Theorem \ref{thm:SegalL} below, can easily be derived from that theorem as we will now show.

For a finite set $L$, let $\Sigma^L$ be the category of finite sets and bijections over $L$ (i.e. bijections of finite sets labelled by $L$). Let $B\Sigma$ be the corresponding $\Gamma$-space. So an $n$-simplex of $B\Sigma^L(A)$ is
\[
U_0 \xrightarrow{\sigma_1} U_1 \xrightarrow{} \dots \xrightarrow{\sigma_n} U_n \xrightarrow{(\alpha,f)} A \times L
\]
where $\sigma_i$'s are bijections and $\alpha$ is the labelling. Let $\FF(L,-) \to B\Sigma^L$ be the map corresponding to the vertex $(id, id) : L \to L \times L$ in $B\Sigma^L(L)$.

\begin{thm}\label{thm:SegalL}
The map $\FF(L,-) \to B\Sigma^L$ is a trivial cofibration in the special $\Gamma$-space localisation of the Reedy model structure on $\GammaSp$.
\end{thm}

We will deduce this theorem from Theorem \ref{thm:Segal}, i.e. the case where $L$ is the one-point set. To this end, consider the functor "product with $L$" and the induced restriction functor $L^* : \GammaSp \to \GammaSp$,
\[
L^*(X)(A) = X(A \times L)
\]
and observe that $B\Sigma^L = L^*(B\Sigma)$.

\begin{prop}\label{prop:Lfunctor}
The functor $L^*$ is a left Quillen functor (for the Reedy, special model structure as above).
\end{prop}

Let us postpone the proof of this proposition, and first explain how the theorem follows.

\begin{proof}[Proof of Theorem \ref{thm:SegalL} assuming Proposition \ref{prop:Lfunctor}]
There is a natural isomorphism
\[
L^* \FF(1,-) = \bigvee_{L} \FF(1, -) \;.
\]
Indeed, for an arbitrary finite set $A$, we have $\FF(1,A) = A_+$, so 
$$L^*\FF(1,-)(A) = \FF(1, A \times L) = (A \times L)_+ = \bigvee_L A_+ = \bigvee_L \FF(1,A) \;. $$
Now consider the diagram
\begin{equation*}
	\begin{tikzpicture}[descr/.style={fill=white}, baseline=(current bounding box.base)] ] 
	\matrix(m)[matrix of math nodes, row sep=2.5em, column sep=2.5em, 
	text height=1.5ex, text depth=0.25ex] 
	{
	\bigvee_L \FF(1,-)  & \FF(L,-) \\
	B \Sigma^L & \\
	}; 
	\path[->,font=\scriptsize] 
		(m-1-1) edge node [auto] {$\simeq$} (m-1-2);
	\path[->,font=\scriptsize] 
		(m-2-1) edge node [auto] {$$} (m-1-2);
	\path[->,font=\scriptsize] 
		(m-1-1) edge node [auto] {$$} (m-2-1);
	\end{tikzpicture} 
\end{equation*}
where the left-hand map is identified with $L^*\FF(1-,) \to L^*B\Sigma$ and so is a trivial cofibration by Proposition \ref{prop:Lfunctor} and Theorem \ref{thm:Segal}, while the horizontal map is a trivial cofibration (in the localised setting) by construction.
\end{proof}

\begin{proof}[Proof of Proposition \ref{prop:Lfunctor}]
First note that $L^*$ has a both a left and right adjoint. Furthermore, $L^*$ obviously preserves degreewise weak equivalences. 

Next, we check that $L^*$ preserves Reedy cofibrations. The generating Reedy cofibrations are of the form
\begin{equation}\label{eq:Lrcof}
\partial \FF(A,-) \times \Delta^n \; \cup \; \FF(A,-) \times \partial \Delta^n \to \FF(A,-) \times \Delta^n
\end{equation}
where $$\partial \FF(A,-) = \bigcup_{A \pto B} \FF(B,-)$$
is the union over (possibly partially defined) surjections. To check that $L^*$ maps (\ref{eq:Lrcof}) to a Reedy cofibration we can work in each fixed simplicial degree separately, so this comes down to checking that for any object $C$ in $\FF$ the group $Aut(C)$ acts freely on the complement of 
\[
\bigcup_{A \pto B} \FF(B, C \times L) \cup \bigcup_{D \hookrightarrow C} \FF(A, D \times L) \hookrightarrow \FF(A, C \times L)
\]
where the unions range over proper partial surjections $A \pto B$ and proper partial injections $D \hookrightarrow C$. 

An element $f :A \to C \times L$ is in this complement if and only if
\begin{itemize}
\item[(i)] $A \xrightarrow{f} C \times L \xrightarrow{\pi_1} C$ is surjective (for otherwise $f$ would factor through some proper injection $D \hookrightarrow C$)
\item[(ii)] $f$ is an injection (for otherwise $f$ would factor through a proper surjection $A \pto B$)
\end{itemize}
Clearly, condition $(i)$ alone ensures that a non-trivial element of $Aut(C)$ cannot fix such an $f$. This proves that $L^*$ is left Quillen for the Reedy model structure.

We still have to prove that $L^*$ sends the localising maps to trivial cofibrations. To simplify the exposition, we focus on the localising map $$\FF(1,-) \vee \FF(1,-) \to \FF(2,-) \; .$$ For an object $A$, the set $L^* \FF(2,-)(A) = \FF(2,A \times L)$ has three kinds of elements: the undefined map, the maps which are defined on just one of the two elements of $2$, and the total maps. So we have a pushout

\begin{equation*}
	\begin{tikzpicture}[descr/.style={fill=white}, baseline=(current bounding box.base)] ] 
	\matrix(m)[matrix of math nodes, row sep=2.5em, column sep=2.5em, 
	text height=1.5ex, text depth=0.25ex] 
	{
	\bigvee_{L^2} \FF(2,A)  & \FF(2,A \times L) \\
	\bigvee_{L^2} \FF^{-}(2,A) & \bigvee_L \FF(1,A) \vee \bigvee_L \FF(1,A)\\
	}; 
	\path[->,font=\scriptsize] 
		(m-1-1) edge node [auto] {$$} (m-1-2)
		(m-2-1) edge node [auto] {$$} (m-2-2)
		(m-2-1) edge node [auto] {$$} (m-1-1)
		(m-2-2) edge node [auto] {$$} (m-1-2);
	\end{tikzpicture} 
\end{equation*}
where $\FF^-(2,A)$ is the set of maps $2 \to A$ that are \emph{not} totally defined. The square is natural in $A$. The map on the left is identified with
\[
\bigvee_{L^2} \FF(2,-) \to \bigvee_{L^2} \FF(1,-) \vee \FF(1,-)
\]
evaluated at $A$, and is a trivial cofibration by definition. Therefore the map on the right, which is easily identified as the map
\[
L^*( \FF(1,-) \vee \FF(1,-)) \to L^* \FF(2,-)
\]
evaluated at $A$, is also a trivial cofibration. The general case, obtained by substituting $2$ with an arbitrary finite set, is treated similarly.
\end{proof}

\bibliography{iekepedro23jan17_arxiv}
\bibliographystyle{abbrv}

\end{document}